\documentclass{article}  
\usepackage{amsmath}
\usepackage{amsthm}
\usepackage{enumerate}
\usepackage{enumitem}
\usepackage{amssymb}
\usepackage{color,epsfig}
\usepackage{fullpage}
\usepackage{algorithm,algorithmicx}
\usepackage[noend]{algpseudocode}

\newtheorem{theorem}{Theorem}
\newtheorem{lemma}[theorem]{Lemma}

\newtheorem{corollary}[theorem]{Corollary}

\newtheorem{remark}{Remark}

\newcommand{\calP}{\mathcal{P}}
\newcommand{\calQ}{\mathcal{Q}}
\newcommand{\calE}{\mathcal{E}}
\newcommand{\NN}{\mathbb{N}}
\newcommand{\RR}{\mathbb{R}}

\renewcommand{\Pr}{\mathbb{P}}

\newcommand{\dist}{\mathop{d}}

\newcommand{\setCompl}[1]{{#1}^{c}}

\usepackage{todonotes}

\usepackage{ifpdf}

\begin{document}

\title{A Bound for the Diameter of Random Hyperbolic Graphs\thanks{A preliminary version of this paper will appear in the Proceedings of the Twelfth Workshop on Analytic Algorithmics and Combinatorics (ANALCO, 2015).}}
\author{Marcos Kiwi\thanks{Depto.~Ing.~Matem\'{a}tica \&
  Ctr.~Modelamiento Matem\'atico (CNRS UMI 2807), U.~Chile. 
  Beauchef 851, Santiago, Chile, Email: \texttt{mkiwi@dim.uchile.cl}. Gratefully acknowledges the support of 
  Millennium Nucleus Information and Coordination in Networks ICM/FIC P10-024F
  and CONICYT via Basal in Applied Mathematics.} \\
\and 
Dieter Mitsche\thanks{Universit\'{e} de Nice Sophia-Antipolis, Laboratoire J-A Dieudonn\'{e}, Parc Valrose, 06108 Nice cedex 02, Email: \texttt{dmitsche@unice.fr}. The main part of this work was performed during a stay at Depto.~Ing.~Matem\'{a}tica \& Ctr.~Modelamiento Matem\'atico (CNRS UMI 2807), U.~Chile, and the author would like to thank them for their hospitality.}}
\date{}

\maketitle


\begin{abstract} \small\baselineskip=9pt Random hyperbolic graphs were
  recently introduced by Krioukov et.~al.~\cite{KPKVB10} as a model
  for large networks.  Gugelmann, Panagiotou, and Peter~\cite{GPP12}
  then initiated the rigorous study of random hyperbolic graphs using
  the following model: for $\alpha> \frac12$, $C\in\RR$, $n\in\NN$,
  set $R=2\ln n+C$ and build the graph $G=(V,E)$ with $|V|=n$ as
  follows: For each $v\in V$, generate i.i.d.~polar coordinates
  $(r_{v},\theta_{v})$ using the joint density function $f(r,\theta)$,
  with $\theta_{v}$ chosen uniformly from $[0,2\pi)$ and $r_{v}$ with
  density $f(r)=\frac{\alpha\sinh(\alpha r)}{\cosh(\alpha R)-1}$ for
  $0\leq r< R$. Then, join two vertices by an edge, if their
  hyperbolic distance is at most $R$. We prove that in the range
  $\frac12 < \alpha < 1$ a.a.s.~for any two vertices of the same
  component, their graph distance is 
  $O(\log^{C_0+1+o(1)}n)$, where $C_0=2/(\frac12-\frac34\alpha+\frac{\alpha^2}{4})$, 
  thus answering a
  question raised in~\cite{GPP12} concerning the diameter of such
  random graphs. As a corollary from our proof we obtain that the
  second largest component has size $O(\log^{2C_0+1+o(1)}n)$, thus answering a
  question of Bode, Fountoulakis and M\"{u}ller~\cite{BFM13}. We also
  show that a.a.s.~there exist isolated components forming a path of
  length $\Omega(\log n)$, thus yielding a lower bound on the size of
  the second largest component.
\end{abstract}


\section{Introduction}\label{sec:intro}
Building mathematical models to capture essential properties of 
  large networks has become an important objective in order 
  to better understand them.
An interesting new proposal in this direction
  is the model of random hyperbolic graphs recently introduced
  by Krioukov et.~al.~\cite{KPKVB10} (see also~\cite{PKBV10}).
A good model should on the one hand replicate the characteristic 
  properties that are observed in real world networks (e.g.,
  power law degree distributions, high clustering and small
  diameter), but on the other hand 
  it should also be susceptible to mathematical analysis.
There are models that partly succeed in the first task but are hard
  to analyze rigorously. 
Other models, like the classical Erd\"os-Renyi $G(n,p)$ model, 
  can be studied mathematically, but fail to capture certain aspects 
  observed in real-world  networks.
In contrast, the authors of~\cite{PKBV10}
  argued empirically and via some non-rigorous methods 
  that random hyperbolic graphs have many of the desired properties. 
Actually, Bogu\~n\'a, Papadopoulos and Krioukov~\cite{BPK10}
  computed explicitly a maximum likelihood fit
  of the Internet graph, convincingly illustrating that this model
  is adequate for reproducing the structure of real networks with high
  accuracy.
Gugelmann, Panagiotou, and Peter~\cite{GPP12} 
  initiated the rigorous study of random hyperbolic graphs. 
They compute exact asymptotic expressions for the maximum degree, the 
  degree distribution (confirming rigorously that the degree sequence 
  follows a power-law distribution with controllable exponent), and 
  also estimated the expectation of the clustering coefficient.

In words, the random hyperbolic graph model
  is a simple variant of the uniform distribution of $n$ vertices within
  a disc of radius $R$ of the hyperbolic plane, where 
  two vertices are connected
  if their hyperbolic distance is at most $R$.
Formally, the random hyperbolic graph model $G_{\alpha,C}(n)$ is defined
  in~\cite{GPP12} as described next:
  for $\alpha> \frac12$, $C\in\RR$, $n\in\NN$, set $R=2\ln n+C$, and 
  build $G=(V,E)$ with vertex set $V=[n]$ as follows:
\begin{itemize}
\item For each $v\in V$, polar coordinates $(r_{v},\theta_{v})$ are 
  generated identically and independently distributed with joint
  density function $f(r,\theta)$, with $\theta_{v}$
chosen uniformly at random in the interval $[0,2\pi)$ and $r_{v}$
  with density:
\begin{align*}
f(r) & = \begin{cases}\displaystyle
   \frac{\alpha\sinh(\alpha r)}{C(\alpha,R)}, &\text{if $0\leq r< R$}, \\
   0, & \text{otherwise},
  \end{cases}
\end{align*}
where $C(\alpha,R)=\cosh(\alpha R)-1$ is a normalization constant.

\item For $u,v\in V$, $u\neq v$, there is an edge with endpoints 
  $u$ and $v$ provided $\dist(r_{u},r_{v},\theta_{u}-\theta_{v})\leq R$,
  where $d=\dist(r,r',\theta-\theta')$ denotes the hyperbolic distance
  between two vertices whose native representation
  polar coordinates are $(r,\theta)$ and $(r',\theta')$, 
  obtained by solving 
\begin{equation}\label{eqn:coshLaw}
\cosh(d) 
 = \cosh(r)\cosh(r')-\sinh(r)\sinh(r')\cos(\theta{-}\theta'). 
\end{equation}
\end{itemize}
The restriction 
  $\alpha>1/2$ and the role of $R$, informally speaking,
  guarantee that the resulting graph has a bounded average degree (depending
  on $\alpha$ and $C$ only).
If $\alpha<1/2$, then the degree sequence is so 
  heavy tailed that this is impossible.

Research in random hyperbolic graphs is in a sense in its infancy.
Besides the results mentioned above, very little else is known.
Notable exceptions are  the emergence and evolution of giant components~\cite{BFM13},
  connectedness~\cite{BFM13b}, results on the global clustering 
  coefficient of the so called binomial model of random hyperbolic graphs~\cite{CF13}, and on the evolution of graphs on more general 
  spaces with negative curvature~\cite{F12}. 

\medskip

\textbf{Notation.} As typical in random graph theory, we shall
consider only asymptotic properties as $n\rightarrow \infty$. We say
that an event in a probability space holds asymptotically almost
surely (a.a.s.) if its probability tends to one as $n$ goes to
infinity. 
We follow the standard conventions concerning 
  asymptotic notation. In particular, we write $f(n)=o(g(n))$, if
  $\lim_{n\to\infty} |f(n)|/|g(n)|=0$. 
We will nevertheless also use $1-o(\cdot)$ when dealing
with probabilities. 
We also say that an event holds 
  \emph{with extremely high probability}, w.e.p., if it occurs
  with probability at least $1-e^{-\omega(\log n)}$. 
Throughout this paper, $\log n$ always denotes the natural logarithm
of $n$. 

The constants $\alpha, C$ used in the model and the constants
$C_0,\delta$ defined below have only one special meaning, other
constants such as $C',C'',c,c_1,c_2,c_3$ change from line to
line. Since we are interested in asymptotic results only, we ignore
rounding issues throughout the paper.

\subsection{Results}
The main problem we address in this work is the 
  natural question, explicitly stated in~\cite[page 6]{GPP12}, that asks to 
  determine the expected diameter of the giant component 
    of a random hyperbolic graph $G$ chosen according to $G_{\alpha,C}(n)$ for $\frac12 < \alpha < 1$. 
We look at this range, since for $\alpha < \frac12$ a.a.s.~a very small central configuration yielding a diameter of at most $3$ exists (see~\cite{BFM13b}): consider a ball of sufficiently small radius around the origin and partition it into $3$ sectors. It can be shown that in each of the sectors a.a.s. there will be at least one vertex, and also that every other vertex is connected to at least one of the three vertices. For $\alpha=\frac12$  the probability of this configuration to exist depends on $C$ (see~\cite{BFM13b}).  For $\alpha > 1$, there exists no giant component (see~\cite{BFM13}); the case $\alpha=1$ 
is a matter requiring further study.

We show (see Theorem~\ref{thm:main}) 
  that for   $\frac12 < \alpha < 1$, a.a.s., for any two vertices of the same
  component, their graph distance is 
$O(\log^{C_0+1+o(1)} n)$, where $C_0=2/(\frac12-\frac34\alpha+\frac{\alpha^2}{4})$. 
To establish our main result we rely on the known,
  and easily established fact, that for the range of $\alpha$ we 
  are concerned with, a graph $G_{\alpha,C}(n)$ has a ``center'' 
  clique whose size is w.e.p.~$\Theta(n^{1-\alpha})$.
Then, we show that, depending on how ``far away'' from the 
  center clique a vertex $Q$ is, 
  there is either a very high or at least non-negligible probability that the vertex
  connects to the center clique through a path of polylogarithmic
  (in $n$) length, or otherwise all paths starting from $Q$ 
  have at most polylogarithmic length.
It immediately follows that two vertices in the same connected
  component, a.a.s., either connect to the center clique 
  through paths of polylogarithmic length, or belong to paths of 
  size at most polylogarithmic.
Either way, a bound on the diameter of $G_{\alpha,C}(n)$ follows.
Rigorously developing the preceding argument requires overcoming 
  significant obstacles, not only technical but also in terms of 
  developing the insight to appropriately define the relevant typical 
  events which are also amenable to a rigorous study of their 
  probabilities of occurrence.
Our main result's proof argument also yields (see Corollary~\ref{cor:main3})
  that the size of the second largest component 
  is $O(\log^{2C_0+1+o(1)} n)$, thus answering a
  question of Bode, Fountoulakis and M\"{u}ller~\cite{BFM13}.
As a complementary result (see Theorem~\ref{thm:main2}), we establish that
  a.a.s.~there exists a component forming an induced path of 
  length $\Theta(\log n)$.
This last result pinpoints a region of 
  hyperbolic space (a constant width band around the origin of
  sufficiently large radius) where it is likely to find 
  a path of length $\Theta(\log n)$.

Another contribution of our work is that it proposes at least two
  refinements and variants of the \emph{breadth exploration process}
  introduced by Bode, Fountoulakis and M\"{u}ller~\cite{BFM13}.
Specifically, we strengthen the method by identifying more involved
  strategies for exploring hyperbolic space, not solely dependent
  on the angular coordinates of its points, and not necessarily 
  contiguous regions of space.
We hope these refinements will be useful in tackling other problems concerning
  the newly minted (and captivating) hyperbolic random graph model. 

\subsection{Organization}
This work is organized as follows. 
In Section~\ref{sec:prelim}
  we introduce the geometric framework and give background results. 
Section~\ref{sec:upper} deals with the upper bound on the diameter 
  as well as the upper bound   on the size of the second largest component.
Section~\ref{sec:lower} is dedicated to the lower bound.

\section{Conventions, background results and preliminaries}\label{sec:prelim}
In this section we first fix some notational conventions.
We then recall, as well as establish, a few results 
  we rely on throughout the following sections. 
All claims in some sense reflect the nature of hyperbolic space, 
  either by providing useful approximations for the angles formed by 
  two adjacent sides of a triangle whose vertices are at
  given distances, or by establishing good approximations for 
  the mass of regions of hyperbolic space obtained by some set 
  algebraic manipulation of balls.
We conclude the section by discussing the de-Poissonization technique
  on which this work heavily relies.

Henceforth, for a point $P$ in hyperbolic plane, 
  we let $(r_{P},\theta_{P})$ denote its polar coordinates 
  ($0\leq r_{P}< R$ and $0\leq \theta_{P}<2\pi$).
The point with polar coordinates $(0,0)$ is called the origin and 
  is denoted by $O$.

By~\eqref{eqn:coshLaw}, the hyperbolic triangle formed by the geodesics 
  between points $A$, $B$, and  $C$, with opposing side segments of length $a$, $b$, and $c$ respectively,
  is such that the angle formed at $C$ is
 (see Figure~\ref{fig:triangHyper}):
\[
\theta_{c}(a,b) = 
\arccos\Big(\frac{\cosh(a)\cosh(b)-\cosh(c)}{\sinh(a)\sinh(b)}\Big).  
\]

\begin{figure}[h]
\begin{center}\ifpdf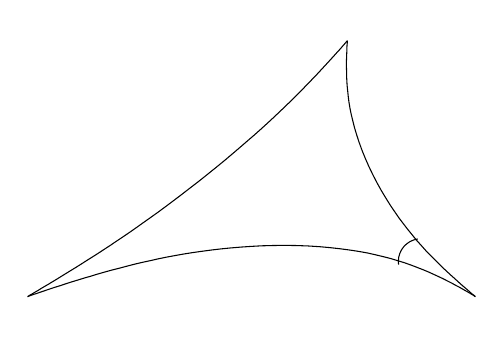\else\input{trianHyper.pstex_t}\fi\end{center}
\caption{Hyperbolic triangle.}\label{fig:triangHyper}
\end{figure}
Clearly, $\theta_{c}(a,b) = \theta_{c}(b,a)$. 

In fact, although some of the proofs hold for a wider range of $\alpha$, 
  we will always assume $\frac12<\alpha<1$, since as 
  already discussed, we are interested in this regime.  
In order to avoid unnecessary repetitions, we henceforth omit this restriction
  from the statement of this article's results.

First, we recall some useful estimates. 
A very handy approximation for $\theta_{c}(\cdot,\cdot)$ 
  is given by the following result.
\begin{lemma}[{\cite[Lemma 3.1]{GPP12}}]\label{lem:aproxAngle}
If $0\leq\min\{a,b\}\leq c\leq a+b$, then
\[
\theta_{c}(a,b) = 2e^{\frac{1}{2}(c-a-b)}\big(1+\Theta(e^{c-a-b})\big).
\]
\end{lemma}
\begin{remark}\label{rem:aproxAngle}
We will use the previous lemma also in this form: let
  $A$ and $B$ be two points such that $r_A,r_B > R/2$, 
  $0 \leq \min\{r_A,r_B\} \leq d:=d(A,B) \leq R$.
Then, taking $c=d$, $a=r_A$, $b=r_B$ in Lemma~\ref{lem:aproxAngle} and noting that $r_{A}+r_{B}\geq R\geq d$, we get
\[
|\theta_A - \theta_B|=\theta_{d}(r_A,r_B) 
  = 2e^{\frac{1}{2}(d - r_A - r_B)}\big(1+\Theta(e^{d-r_A-r_B})\big).
\]
Note also that for fixed $r_A,r_B > R/2$, $|\theta_A - \theta_B|$ is increasing as a function of $d$ (for $d$ satisfying the constraints). 
Below, when aiming for an upper bound, we always use $d=R$.
\end{remark}
Throughout, we will need estimates for measures of 
  regions of the hyperbolic plane, and more specifically, for regions obtained by performing some set algebra
  involving a couple of balls.
Hereafter, for a point $P$ of the hyperbolic plane, $\rho_{P}$ will be 
  used to denote the radius 
  of a ball centered at $P$, and $B_{P}(\rho_{P})$ to denote 
  the closed ball of radius $\rho_{P}$ centered at $P$.
Also, we denote by $\mu(S)$ the probability measure of a set $S$, i.e.,
\[
\mu(S) = \int_{S}f(r,\theta)drd\theta.
\]

We collect now a few results for such measures.
\begin{lemma}[{\cite[Lemma~3.2]{GPP12}}]\label{lem:muBall} 
For any $0\leq \rho_{O}\leq R$, $\mu(B_{O}(\rho_{O})) =
  e^{-\alpha(R-\rho_{O})}(1+o(1))$.
\end{lemma}
The following result gives an approximation for the mass
  of the intersection of two balls of ``reasonable size'', one of which
  being centered at the origin $O$ and the other one containing the origin.
\begin{lemma}\label{lem:muBallInterGen}
Let $C_{\alpha}=2\alpha/(\pi(\alpha-\frac12))$.
For $r_{A}\leq \rho_{A}$ and $\rho_{O}+r_{A}\geq \rho_{A}$,
\[
\mu(B_{A}(\rho_{A})\cap B_{O}(\rho_{O})) 
  = C_{\alpha}e^{-\alpha(R-\rho_{O})-\frac{1}{2}(\rho_{O}-\rho_A+r_A)} 
        + O(e^{-\alpha(R-\rho_{A}+r_{A})}).
\]
\end{lemma}
\begin{proof}
We want to bound (see Figure~\ref{fig:interiorInter}): 
\begin{figure}[h]
\begin{center}\ifpdf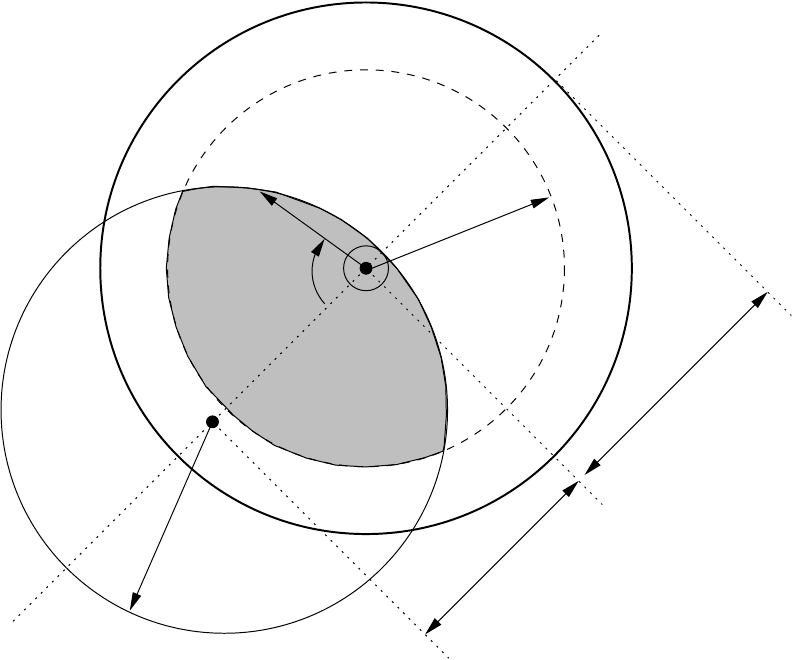\else\input{intersect.pstex_t}\fi\end{center}\caption{the shaded area corresponds to $B_{A}(\rho_A)\cap B_{O}(\rho_0)$.}\label{fig:interiorInter}
\end{figure}
\[
\mu(B_{A}(\rho_{A})\cap B_{O}(\rho_{O})) 
   = \mu(B_{O}(\rho_{A}-r_A))+2\int_{\rho_{A}-r_{A}}^{\rho_{O}}\int_{0}^{\theta_{\rho_A}(r,r_{A})} \frac{f(r)}{2\pi} d\theta dr.
\]
Relying on the approximation of $\theta_{\rho_A}(r,r_{A})$ given 
  by Lemma~\ref{lem:aproxAngle}, 
\begin{align*}
& \mu(B_{A}(\rho_{A})\cap B_{O}(\rho_{O})) \\
& \qquad  = \mu(B_{O}(\rho_{A}-r_A))
      +\frac{\alpha e^{\frac{1}{2}(\rho_{A}-r_A)}}{\pi C(\alpha,R)} 
  \times
      \int_{\rho_{A}-r_{A}}^{\rho_{O}} 
      (e^{(\alpha-\frac12)r}-e^{-(\alpha+\frac12)r})(1+\Theta(e^{\rho_{A}-r-r_{A}}))dr. 
\end{align*}
  We will first solve the integral without the error term $\Theta(e^{\rho_{A}-r-r_{A}})$. Since $C_{\alpha}=2\alpha/(\pi(\alpha-\frac12))$, 
  and recalling that $C(\alpha,R)=\cosh(\alpha R)-1=
    \frac{1}{2}e^{\alpha R}+O(e^{-\alpha R})$,
  for this part we obtain
\begin{align*}
& \frac{\alpha e^{\frac{1}{2}(\rho_{A}-r_A)}}{\pi C(\alpha,R)}
      \Big(\frac{e^{(\alpha-\frac12)\rho_{O}}-e^{(\alpha-\frac12)(\rho_A-r_{A})}}{\alpha-\frac12}+O(1)\Big) \\
& \qquad  =  
        C_{\alpha}
        \Big(e^{-\alpha(R-\rho_{O})-\frac{1}{2}(\rho_{O}-\rho_A+r_A)}-e^{-\alpha(R-\rho_A+r_A)}
+\Theta(e^{\frac{1}{2}(\rho_A-r_A)-\alpha R})\Big) \\
& \qquad = 
 C_{\alpha}
 \Big(e^{-\alpha(R-\rho_{O})-\frac{1}{2}(\rho_{O}-\rho_A+r_A)}+O(e^{-\alpha(R-\rho_A+r_A)})\Big),
\end{align*}
where the last identity is because 
  $e^{\frac{1}{2}(\rho_{A}-r_{A})-\alpha R}=O(e^{-\alpha(R-\rho_{A}+r_{A})})$ for 
  $\alpha > \frac12$.
Now, for the error term, since $\alpha<\frac32$, we obtain
\[
\frac{\alpha e^{\frac{3}{2}(\rho_{A}-r_A)}}{\pi C(\alpha,R)}
      \Big(\frac{\Theta(e^{(\alpha-\tfrac{3}{2})\rho_{O}}-e^{(\alpha-\tfrac{3}{2})(\rho_A-r_{A})})}{\alpha - \tfrac{3}{2}}  + O(1) \Big) 
   = \frac{\alpha e^{\frac{3}{2}(\rho_{A}-r_A)}}{\pi C(\alpha,R)}O(1)
   = O(e^{-\alpha(R - \rho_A + r_A)}).
\]
Applying Lemma~\ref{lem:muBall} to $\mu(B_{O}(\rho_{A}-r_A))$, the desired conclusion follows.
\end{proof}

{From} now on, denote $R_{0}:=R/2$ and 
  $R_{i}:=Re^{-\alpha^{i}/2}$ for $i\geq 1$. 
Observe that $R_{i-1}\leq R_{i}$ for all $i$.
Next, we show that for $R/2<r_A\leq R-O(1)$, 
  a constant fraction of the mass of the intersection
  of the balls $B_{A}(R)$ and $B_{O}(R_{i})$ is explained by 
  the mass of the intersection of $B_{A}(R)$ with just the band
  $B_{O}(R_i)\setminus B_{O}(R_{i-1})$.

\begin{lemma}\label{lem:simplebd2}
Let $\xi>0$ be some constant and
 $i \geq 1$ be such that $R_i < R - \xi$. 
If  $R_{i} < r_{A} \leq R_{i+1}$, then 
\[
\mu(B_{A}(R) \cap \big(B_{O}(R_{i}) \setminus B_{O}(R_{i-1})\big)) 
   = \Omega(\mu(B_{A}(R) \cap B_{O}(R_{i}))),
\]
where the constant $C'=C'(\xi,\alpha)$ hidden inside 
  the asymptotic notation can be assumed
  nondecreasing as a function of $\xi$.
\end{lemma}
\begin{proof}
     Since $R_{i-1}\leq R_{i}$, we have that 
  $B_{O}(R_{i-1})\subseteq B_{O}(R_{i})$, so the left hand side of the 
  stated identity can be re-written as 
  $\mu(B_{A}(R)\cap B_{O}(R_{i}))-\mu(B_{A}(R)\cap B_{O}(R_{i-1}))$.
Now, observe that by Lemma~\ref{lem:muBallInterGen} applied with 
  $\rho_A=R$, $\rho_O=R_i$ (so $R_i + r_A > 2R_i \geq R$), 
\begin{equation}\label{eq:bigball}
\mu(B_{A}(R)\cap B_{O}(R_i)) 
  = (1+o(1))C_{\alpha}e^{-\alpha(R-R_i)-\frac{1}{2}(R_i-R+r_A)}, 
\end{equation}
where $C_{\alpha}=2\alpha/(\pi(\alpha-\frac12)$.
By the same argument, this time applied with $\rho_A=R$, $\rho_O=R_{i-1}$, 
we also have
\begin{equation}\label{eq:smallball}
\mu(B_{A}(R)\cap B_{O}(R_{i-1})) 
  = (1+o(1))C_{\alpha}e^{-\alpha(R-R_{i-1})-\frac{1}{2}(R_{i-1}-R+r_A)}. 
\end{equation}
It suffices to show that the ratio $\rho(R,i)$ between the right hand 
  sides of the expressions in~\eqref{eq:bigball} 
  and~\eqref{eq:smallball} is at least a constant.
We claim that the assertion regarding $\rho(R,i)$ holds.
Indeed, note that
\[
\rho(R,i) = (1+o(1))e^{(\alpha-\frac12)(R_{i}-R_{i-1})} =
  (1+o(1))e^{(\alpha-\frac12)R(e^{-\alpha^{i}/2}-e^{-\alpha^{i-1}/2})}, 
\]
and let $i_{1}=O(1)$ be large enough so that 
  $1-\alpha\geq \alpha^{i_1-1}/2$.
Since $\frac12<\alpha<1$, 
  $\beta:=e^{-\alpha^{i_{1}}/2}-e^{-\alpha^{i_{1}-1}/2}>0$ is a 
  constant such that if $i\leq i_1$, then 
  $\rho(R,i)\geq\rho(R,i_{1})=(1+o(1))e^{(\alpha-\frac12)\beta R}$.
Thus, the claim holds for $i\leq i_1$.

Using now $1-x \leq e^{-x} \leq 1-x+x^2/2$, by our choice of $i_1$, and recalling
  that $\alpha<1$, 
\[
R\Big(e^{-\alpha^{i}/2}-e^{-\alpha^{i-1}/2}\Big)
  \geq R\frac{\alpha^{i-1}}{2}\Big(1-\alpha-\frac{\alpha^{i-1}}{4}\Big) 
  \geq R\frac{\alpha^{i-1}}{2}\cdot\frac{1-\alpha}{2}.
\]
Since by assumption $R_i < R-\xi$, we also have 
  $\alpha^i \geq \frac{2\xi}{R}$, and thus  
  $\rho(R,i)\geq (1+o(1))e^{(\alpha-\frac12)(1-\alpha)\xi/(2\alpha)}$,
  finishing the proof.
\end{proof}

In order to simplify our proofs, we will make use of a technique known
as de-Poissonization, which has many applications in geometric
probability (see~\cite{Penrose} for a detailed account of the
subject). Throughout the paper we work with a Poisson point process on
the hyperbolic disc of radius $R$ and denote its point set by
$\calP$. Recall that $R=2\log n+C$ for $C \in \RR$. The
intensity function at polar coordinates $(r,\theta)$ for $0 \leq r < R$ 
  and $0 \leq \theta < 2\pi$ is equal to
\[
g(r,\theta) := \delta e^{R/2}f(r,\theta)
  = \delta e^{R/2}
  \frac{\alpha\sinh(\alpha r)}{2\pi C(\alpha,R)}
\]
with $\delta=e^{-C/2}$. 
Note that $\int_{r=0}^R \int_{\theta=0}^{2\pi} g(r,\theta) d\theta dr 
  = \delta e^{R/2}=n$, 
  and thus  $\mathbb{E}|{\calP}|=n.$

  The main advantage of defining $\calP$ as a Poisson point process is
  motivated by the following two properties: the number of points of
  $\calP$ that lie in any region $S \cap B_O(R)$ follows a Poisson
  distribution with mean given by $\int_S g(r,\theta) drd\theta=n
  \mu(S \cap B_O(R))$, and the number of points of $\calP$ in disjoint
  regions of the hyperbolic plane are independently distributed.
  Moreover, by conditioning $\calP$ upon the event $|{\calP}|=n$, we
  recover the original distribution.   
Note that for any $k$, by standard estimates for the Poisson 
  distribution, we have 
   \begin{align}\label{lem:Poisson1}
  \Pr(|{\calP}|=k) \leq \Pr(|{\calP}|=n)=O(1/\sqrt{n}).
  \end{align}
In some of the arguments below, we will add a set of vertices $\calQ$ 
  of size $|\calQ|=m$, that are all chosen independently according to 
  the same probability distribution as every vertex in 
  $G_{\alpha,C}(n)$, we add them to $\calP$ and consider $\calP \cup \calQ$. 
Throughout this work all vertices named by $Q$, 
  or in case there are more, $Q_1,\ldots,Q_m$ are vertices added in this way.
  For $m=o(n^{1/2})$, by Stirling's approximation, we have
\begin{equation}\label{lem:Stirling}
\Pr(|{\calP}|=n-m)= e^{-n}\frac{n^{n-m}}{(n-m)!} 
  = e^{-m}\Big(\frac{n}{n-m}\Big)^{n-m}\Theta(n^{-1/2}) 
  = e^{-m}\Big(1+\frac{m}{n-m}\Big)^{n-m}\Theta(n^{-1/2}). 
\end{equation} 
Using that $e^{x/(1+x)}\leq 1+x$ if $x>-1$ and recalling that 
  $m=o(\sqrt{n})$, we get 
\[
\Pr(|{\calP}|=n-m)\geq e^{-m}e^{(n-m)\frac{m}{n}}\Theta(n^{-1/2})  
  = e^{-o(1)}\Theta(n^{-1/2})=\Omega(1/\sqrt{n}).
\]
Thus, for such an $m$, by~\eqref{lem:Stirling},
  \begin{align}\label{lem:Poisson}
    \Pr(|{\calP}|=n-m)=\Theta(1/\sqrt{n}).
  \end{align}
By conditioning upon $|{\calP}|=n-m$, we recover the original 
  distribution of $n$ points. 
Moreover, for any $m=o(\sqrt{n})$, any event
  holding in $\calP$ or $\calP \cup \calQ$
  with probability at least $1-o(f_n)$ must hold in
  the original setup with probability at least $1-o(f_n \sqrt n)$, and
  in particular, any event holding with probability at least
  $1-o(1/\sqrt{n})$ holds asymptotically almost surely (a.a.s.), that
  is, with probability tending to $1$ as $n \to \infty$, in the
  original model. We identify below points of $\calP$ with
  vertices. We prove all results below for the Poisson model and then
  transform the results to the original model.

\section{The upper bound on the diameter}\label{sec:upper}
In this section we prove the main result of this work, a polylogarithmic upper bound in $n$ on the
diameter $G_{\alpha,C}(n)$ that holds asymptotically almost surely.

\medskip
We begin by showing that there cannot be long paths with all vertices being 
  close to the boundary (i.e.,~at distance $R-O(1)$ from the origin $O$).
\begin{lemma}\label{lem:nolongpath}
Let $\xi > 0$. 
Let $Q \in \calQ$ be a vertex added, 
  and suppose $r_Q > R-\xi$.
There is a constant $C'=C'(\xi,C)$ such that 
  with probability at least $1-o(n^{-3/2})$, there is no path
  $Q=:v_0,v_1,\ldots,v_k$, with $k \geq C' \log n$ and 
  $r_{v_i} > R-\xi$ for $i=1,\ldots,k$.
\end{lemma}
\begin{proof}
First, note that by  Remark~\ref{rem:aproxAngle}, 
  for two vertices $v_i,v_j$ with
  $r_{v_i}, r_{v_j} > R-\xi$ and $d(v_i,v_j) \leq R$, we must have
  $|\theta_{v_i} - \theta_{v_j}| \leq (1+o(1))2e^{-R/2+\xi}=C''/n$ for
  $C''=(2+o(1))e^{-C/2 + \xi}$. 
Partition the disc of radius $R$ into
  $\Theta(n)$ equal sized sectors of angle $2\pi C''/n$ 
  and order them
  counterclockwise. Note that any  path containing only vertices $v$ with $r_v > R-\xi$ satisfies the following property: if it contains a vertex $v_i$ in a sector $S_i$
  and a vertex $v_j$ in a sector $S_j$, either all sectors between $S_i$ and $S_j$ in the counterclockwise ordering or all sectors between $S_i$ and $S_j$ in the clockwise ordering have to contain at least one vertex from the path. 
By symmetry, for each sector, 
  the expected number of vertices inside the
  sector is $C''$, and therefore this is also an upper bound for the
  expected number of vertices $v$ inside this sector with the
  additional restriction of $r_v > R-\xi$. Thus, the probability of
  having at least one vertex inside a sector is at most $1-e^{-C''} < 1.$ 
The probability to have at least one vertex in $C''' \log n$
  given consecutive sectors starting from the sector that contains $Q$ is thus at most $2(1-e^{-C''})^{C'''\log n} =
  o(n^{-3/2})$ for $C'''=C'''(C'')$ sufficiently large. Thus, with probability at least $1-o(n^{-3/2})$, any such path starting at $Q$ contains only 
  vertices from at most $C''' \log n$ sectors. By Chernoff bounds together with a union bound, for $C'''$ large enough, with probability at least $1-o(n^{-3/2})$, any consecutive set of $C''' \log n$ sectors contains $O(\log n)$ vertices, and the statement follows.
\end{proof}

Because of the peculiarities of hyperbolic space, two points
  close to the boundary $R$ that are within distance at most $R$ from
  each other subtend a small angle at the origin. 
In fact, the angle is smaller the closer both points are to the boundary. 
This fact, coupled with the previous lemma, implies that with high probability
  there is no path all of whose vertices are within a constant distance of 
  the boundary whose extreme vertices subtend a large angle at the origin.
Formally, we have the following corollary:
\begin{corollary}\label{cor:nolongpath}
Let $\xi > 0$. 
Let $Q \in \calQ$ be a vertex added, 
  and suppose $r_Q > R-\xi$, and let   $Q=:v_0,v_1,\ldots,v_k$ be a path with $r_{v_i} > R-\xi$ for $i=1,\ldots,k$.
Then,  with probability at least $1-o(n^{-3/2})$, for any $0 \leq i \leq j  \leq k$,
$$
|\theta_{v_i}-\theta_{v_j}| = O(\log n/n).
$$
\end{corollary}
\begin{proof}
By Lemma~\ref{lem:nolongpath}, with probability at least $1-o(n^{-3/2})$, $k=O(\log n)$. By the proof of Lemma~\ref{lem:nolongpath}, for any two $v_i,v_j$ with
  $r_{v_i}, r_{v_j} > R-\xi$ and $d(v_i,v_j) \leq R$, we must have
  $|\theta_{v_i} - \theta_{v_j}| = O(1/n)$. The corollary follows by the triangle inequality for angles.
\end{proof}
Recall that $C_0=2/(\frac12 - \frac34 \alpha + \frac{\alpha^2}{4})$ and note that $C_0 > 8$. Define throughout the whole section 
  $i_{0}:=\log_{\alpha} ((2C_{0}\log R)/R)$.  Note that
  $R_{i_0}=Re^{-\alpha^{i_0}/2}=R(1-(1+o(1))\frac{\alpha^{i_0}}{2})=R-(1+o(1))C_0\log R$.

As already pointed out, the farther from the origin two vertices within distance at most $R$ from each other are, the smaller the angle they subtend at the origin.
In particular, for any two adjacent vertices $v_i,v_{i+1}$ with $r_{v_i},r_{v_{i+1}} > R_{i_0}$, by Remark~\ref{rem:aproxAngle}, 
\begin{equation}\label{eq:anglebound}
|\theta_{v_i} - \theta_{v_{i+1}}| \leq (2+o(1))e^{\frac12(R-r_{v_i}-r_{v_{i+1}})} \\
  \leq (2+o(1))e^{\frac12(R-2R_{i_0})} 
  =    e^{-\frac12 R}R^{C_0(1+o(1))}=O(\tfrac{1}{n}\log^{C_0+o(1)} n). 
\end{equation}

For a vertex $v \in \calP$ with $r_v > R_{i_0}$, let $\ell$ be such that 
  $R_{\ell} < r_{v} \leq R_{\ell+1}$.
Define a \emph{center path} from $v$ to be a sequence of vertices 
  $v=:w_0,w_1,\ldots,w_s$ such that $d(w_i,w_{i+1}) \leq R$, 
  $R_{\ell-i} < r_{w_i} \leq R_{\ell-i+1}$ for $0 \leq i \leq s-1$, and 
  $R_{i_0-1} < r_{w_s} \leq R_{i_0}$. In words,  $w_0,w_1,\ldots,w_s$ is a center path 
  from $v$ provided $w_0=v\in B_{O}(R_{\ell+1})\setminus B_{O}(R_{\ell})$,
  $w_{s}\in B_{O}(R_{i_0})$,
  and for every $0<i<s$ the vertex $w_i$ is in the band
  $B_{O}(R_{\ell-i+1})\setminus B_{O}(R_{\ell-i})$ around the origin; note that 
  each band is closer to the origin than the previous one.

We now establish the fact that 
  if a vertex is ``far'' from the origin but some constant distance 
  away from the boundary, then its probability of connecting
  via a center path
  to a vertex inside $B_{O}(R_{i_0})$ is 
  at least a constant.
\begin{lemma}\label{lem:centerpath}
Let $\xi' > 0$ and $Q \in \calQ$ be a vertex added.
Suppose that $R_{\ell} < r_Q \leq R_{\ell+1}$ for $\ell \geq i_0$ and
  $R_{\ell} \leq R-\xi'$.  
Then, with probability bounded away from $1$, there exists no center path 
  $Q=:\omega_0,\ldots,\omega_s$
  from~$Q$.
\end{lemma}
\begin{proof}
Denote by $\calE_0$ the event that there exists one vertex of $\calP$ that 
  belongs to $B_{w_0}(R)\cap (B_{O}(R_{\ell})\setminus B_{O}(R_{\ell-1}))$. 
By Lemma~\ref{lem:simplebd2}, 
\[
\mu(B_{w_0}(R)\cap (B_{O}(R_{\ell})\setminus B_{O}(R_{\ell-1}))) 
  = \Omega(\mu(B_{w_0}(R)\cap B_O(R_{\ell}))).
\]
By  Lemma~\ref{lem:muBallInterGen} (applied with
  $\rho_A=R$, $\rho_O=R_{\ell}$, and $R_{\ell} < r_A=r_{w_0} \leq R_{\ell+1}$), we have
\begin{align*}
\mu(B_{w_0}(R)\cap B_O(R_{\ell})) 
  = \frac{2\alpha}{\pi(\alpha - \frac12)}\Big(e^{-(\alpha - \frac12)(R-R_{\ell})-\frac12 r_{w_0}}+O(e^{-\alpha r_{w_0}}) \Big) 
  =\Theta(e^{-(\alpha - \frac12)(R-R_{\ell})-\frac12 r_{w_0}}),
\end{align*}
  where the last identity follows because 
   $r_{w_0}>  R-R_{\ell}$.
 Now, since $1-e^{-x} \leq x$, we have $R-R_{\ell}=R(1-e^{-\alpha^{\ell}/2}) \leq R\alpha^{\ell}/2$, and since for $x=o(1)$, $e^{-x}=(1+o(1))(1-x)$, we also have $r_{w_0} \leq R_{\ell+1}=Re^{-\alpha^{\ell+1}/2}= R(1-(1+o(1))\alpha^{\ell+1}/2)$. Thus
\[
\Pr(\setCompl{\calE_0}) = \exp(-\Omega(ne^{-(\alpha-\frac12)(R-R_{\ell})-\frac12 r_{w_0}}))
   =\exp(-\Omega(ne^{-(\alpha-\frac12)\frac{R}{2}\alpha^{\ell}-(1-(1+o(1))\frac{R}{4}\alpha^{\ell+1})})).
\]
  Recalling that $e^{-R/2}=\Theta(1/n)$, we have 
   \begin{equation}\label{eq:nopath}
   \Pr(\setCompl{\calE_0}) 
    = \exp(-\Omega(e^{-(\alpha-\frac12)\frac{R}{2}\alpha^{\ell} + (1+o(1))\frac{R}{4}\alpha^{\ell+1}})). 
   \end{equation}
  Note that if $\calE_0$ holds, then a vertex inside the desired region is found. Assuming the existence of a vertex $w_i$ with $R_{\ell-i} < r_{w_i} \leq R_{\ell-i+1}$, we continue inductively for $i=1,\ldots,s-1=\ell-i_0$ in the same way: we define the event $\calE_i$ that there exists one element of $\calP$ that belongs to  $B_{w_i}(R)\cap (B_{O}(R_{\ell-i})\setminus B_{O}(R_{\ell-i-1}))$. By the same calculations (noting that Lemma~\ref{lem:muBallInterGen}  can still be applied and $R_{\ell-i}+r_{w_i} > R$ still holds), we obtain
     $$\Pr(\setCompl{\calE_i}) = \exp(-\Omega(e^{-(\alpha-\frac12)\frac{R}{2}\alpha^{\ell-i} + (1+o(1))\frac{R}{4}\alpha^{\ell-i+1}})).
   $$
 Denote by $\cal{C}$ the event of not having a center path, we have thus by independence of the events
 $$
\Pr({\cal{C}}) \leq \Pr(\setCompl{\calE_0}) +\sum_{i=1}^{s-1}\Pr(\setCompl{\calE_i} | \calE_0,\ldots,\calE_{i-1}) = \sum_{i=0}^{s-1}\Pr(\setCompl{\calE_i}). 
 $$
Hence,
\begin{align*}
& \Pr({\cal{C}}) \leq\sum_{i=0}^{\ell-i_0}\exp(-\Omega(e^{-(\alpha-\frac12)\frac{R}{2}\alpha^{\ell-i} + (1+o(1))\frac{R}{4}\alpha^{\ell-i+1}}))
 \leq \sum_{i=0}^{\ell-i_0}e^{-\Omega(1-(\alpha-\tfrac12)\frac{R}{2}\alpha^{\ell-i}+(1+o(1))\frac{R}{4}\alpha^{\ell-i+1})} \\
& \qquad \leq \sum_{i=0}^{\ell-i_0}e^{-\Omega(1+(1+o(1))\frac{R}{4}\alpha^{\ell-i}(1-\alpha))} 
  \leq \sum_{i=0}^{\ell-i_0}e^{-C'-(1+o(1))C'\frac{R}{4}\alpha^{\ell-i}(1-\alpha)}, 
\end{align*}
where $C'=C'(\xi') > 0$ is the constant hidden in the asymptotic notation of Lemma~\ref{lem:simplebd2}.
Clearly, the closer $Q$ is to the boundary, the more difficult it is to find a center path, and we may thus assume that $R_{\ell } < r_{Q} \leq R_{\ell+1}$ is such that $R_{\ell+1}>R-\xi'$. 
 Then, noting that $R_{\ell+1}=Re^{-\alpha^{\ell+1}/2}>R-\xi'$ implies that $\ell=\log_{\alpha} \big((1+o(1))2\xi'/R\big)$ must hold. Plugging this into the previous sum we 
get 
$$
 \Pr({\cal{C}}) \leq e^{-C'}\sum_{i \geq 0} e^{-(1+o(1))C'\frac{\xi'(1-\alpha)}{2\alpha^i}}.
$$
Clearly, since $\alpha^{-1} > 1$, the sum converges. 
Note that the constant $C'$ coming from Lemma~\ref{lem:simplebd2} 
  is nondecreasing as a function of $\xi'$. 
Hence, by choosing $\xi'=\xi'(C',\alpha)$ big enough, the sum is less than 
  $1$, and the statement follows.
\end{proof}
Define $j_0=j_0(\alpha)\geq 1$ to be a constant sufficiently large
  so that $e^{-\alpha^{j}/2}\leq 1-(1-\frac{1-\alpha}{2})\frac{\alpha^{j}}{2}$ 
  for $j \geq j_0$ (note that such $j_0$ exists because 
  $e^{-x}\leq 1-x+x^2$ if $|x|\leq 1$).

The following lemmas will show that for vertices $v$ with 
  $r_v \leq R_{i_0}$ the probability of not connecting to a vertex 
  in $B_O(R/2)$  can be bounded by a much smaller expression
  than the analogous 
  probability bound for vertices close to the boundary $R$
  given by Lemma~\ref{lem:centerpath}.
First, we consider the case of vertices in the band 
  $B_{O}(R_{j})\setminus B_{O}(R_{j-1})$ for $j_0\leq j\leq i_0$,
  and show that w.e.p.~they have neighbors in the next band closer
  to the origin.
Formally, we establish the following result, whose proof argument is 
  reminiscent of that of Lemma~\ref{lem:centerpath}, although the 
  calculations involved are different.
\begin{lemma}\label{lem:farcenter}
Let $Q \in \calQ$ be a vertex added such that 
  $R_{j} < r_Q \leq R_{j+1}$  with $j_0 \leq j \leq {i_0}$. 
W.e.p., 
$$
\big( B_{Q}(R) \cap (B_O(R_j) \setminus B_O(R_{j-1})\big) 
  \cap \calP \neq \emptyset.
$$
\end{lemma}
\begin{proof}
As in the previous proof, by Lemma~\ref{lem:simplebd2}, 
\begin{equation}\label{eq:asbefore}
 \mu(B_{Q}(R)\cap (B_{O}(R_j)\setminus B_{O}(R_{j-1}))) 
  =\Omega(\mu(B_{Q}(R)\cap B_O(R_j))), 
\end{equation}
and also as before, by  Lemma~\ref{lem:muBallInterGen} (applied with
  $\rho_A=R$, $\rho_O=R_{j}$, and $R_{j} < r_A=r_Q \leq R_{j+1}$), we have
  $$
  \mu(B_{Q}(R)\cap B_O(R_j)) =\Theta(e^{-(\alpha - \frac12)(R-R_{j})-\frac12 r_{Q}}).
$$
Again, $R-R_{j}=R(1-e^{-\alpha^j/2}) \leq R\alpha^{j}/2$, and since by assumption 
  $e^{-\alpha^{j}/2} \leq 1 - (1-\frac{1-\alpha}{2})\frac{\alpha^{j}}{2}$,  
  we have $r_{Q} \leq R_{j+1}=Re^{-\alpha^{j+1}/2}\leq R(1-(1-\frac{1-\alpha}{2})\frac{\alpha^{j+1}}{2})$. 
Thus,
\[
\mu(B_{z}(R)\cap B_O(R_j)) 
  = \Omega(e^{-(\alpha-\frac12)\frac{R}{2}\alpha^{j}-(1-(1-\frac{1-\alpha}{2})\frac{R}{4}\alpha^{j+1})})
  = \Omega(e^{-R/2}e^{\frac{R}{2}\alpha^j(\frac12-\frac34\alpha+\frac{\alpha^2}{4})}).
\]
Note that $\frac12 - \frac34 \alpha + \frac{\alpha^2}{4} > 0$
  for $\frac12 < \alpha < 1$, so the last displayed expression  
  is clearly decreasing in $j$. 
By plugging in our upper bound $j=i_0=\log_{\alpha} (2 C_0 \log R/R)$,
  we get with our choice of $C_0=2/(\frac12-\frac34\alpha+\frac{\alpha^2}{4})$, 
\[
  \mu(B_{Q}(R)\cap B_O(R_j)) = 
  \Omega\Big(\frac{R^{2}}{e^{R/2}}\Big)=\Omega((\log n)^{2}/n).
\]
Hence, the expected number of vertices in 
  $B_{Q}(R)\cap B_O(R_j)$ is $\Omega(\log^{2} n)$, and by Chernoff bounds 
  for Poisson random variables (see~\cite[Theorem~A.1.15]{AlonSpencer}), w.e.p.~there are at least $\Omega(\log^{2} n)$ vertices in this region. By~\eqref{eq:asbefore} the statement follows.
\end{proof}
Next, we show the analogue of the preceding lemma, but
  for vertices within $B_{O}(R_{j_0})$ and a somewhat different choice of
  concentric bands. Indeed,
  for vertices $v \in \calP$ with $R/2 < r_v \leq R_{j_0}$ 
  we modify the definition of $R_i$: since 
  $j_0=O(1)$, there exists some $\frac12< c< 1$ such that 
  $R_{j_0}=Re^{-\alpha^{j_0}/2}=: cR$. 
Let $T \geq 1$ be the largest integer such that 
  $c-\frac{T}{2}(1-c)(1-\alpha)>\frac12$. 
Define now the new bands to be $R'_0:=cR$, and for any $i=1,\ldots,T$, 
  define $R'_i:=R(c-\frac{i}{2} (1-c)(1-\alpha))$. 
Note in particular, that $R'_i\geq R/2$ for all $i$ in the 
  range of interest.
We have the following result.
\begin{lemma}\label{lem:closecenter}
Let $Q \in \calQ$ be a vertex added.
\begin{itemize}
\item If  $R'_j < r_Q \leq R'_{j-1}$ for some 
  $1 \leq j \leq T-1$, then w.e.p., we have
\[
\big(B_{Q}(R)\cap (B_O(R'_j) \setminus B_O(R'_{j+1})\big) \cap \calP 
  \neq \emptyset.
\]
\item If  $R'_T < r_Q \leq R'_{T-1}$, then w.e.p.~we have
\[
\big(B_{Q}(R)\cap B_O(R/2)\Big) \cap \calP \neq \emptyset.
\]
\end{itemize}
\end{lemma}
\begin{proof}
First, assume $1 \leq j \leq T-1$. 
By  Lemma~\ref{lem:muBallInterGen} (applied with
  $\rho_A=R$, $\rho_O=R'_{j}$, and $R'_{j} < r_A =r_Q\leq R'_{j-1}$), since $R'_j+r_Q > R$ still holds, we have
\[
\mu(B_{Q}(R)\cap B_O(R'_j)) =\Theta(e^{-(\alpha - \frac12)(R-R'_{j})-\frac12 r_{Q}})
  = \Theta(e^{-\frac{R}{2}(2\alpha-1)(1-c+\frac{j}{2}(1-c)(1-\alpha)) - \frac12 r_Q}),
\]
     and also 
\[
\mu(B_{Q}(R)\cap B_O(R'_{j+1})) 
  = \Theta(e^{-\frac{R}{2}(2\alpha-1)(1-c+\frac{j+1}{2}(1-c)(1-\alpha))-\frac12 r_Q}).
\]
Thus,  
\[
\mu(B_{Q}(R)\cap B_O(R'_{j+1})) 
  = \Theta(e^{-\frac{R}{4}(2\alpha-1)(1-c)(1-\alpha)})\mu(B_{Q}(R)\cap B_O(R'_j)),
\]
so  we get
\begin{equation}\label{eq:inclusion}
\mu(B_{Q}(R)\cap (B_O(R'_j) \setminus B_O(R'_{j+1}))) 
   = (1+o(1)) \mu(B_{Q}(R)\cap B_O(R'_j)).
\end{equation}
Now, since $r_Q \leq R'_{j-1}$, 
\begin{align*}
& \mu(B_{Q}(R)\cap B_O(R'_j)) 
  = \Omega\big(e^{-\frac{R}{2}(2\alpha-1)(1-c+\frac{j}{2}(1-c)(1-\alpha))-\frac12 R'_{j-1}}\big)\\ 
&  = \Omega\Big(e^{-\frac{R}{2}\big((2\alpha-1)(1-c+\frac{j}{2}(1-c)(1-\alpha)) + c - \frac{j-1}{2}(1-c)(1-\alpha)\big)}\Big)
   = \Omega\Big(e^{-\frac{R}{2}\big(-j(1-c)(1-\alpha)^2 + (1-c)(2\alpha-1+\frac{1-\alpha}{2})+c\big)}\Big).
\end{align*}
Clearly, the bigger $j$, the bigger the last displayed term is.
Thus, we plug in the smallest possible value 
  $j=1$, and together with~\eqref{eq:inclusion} we obtain
\begin{align*}
& \mu(B_{Q}(R)\cap (B_O(R'_j) \setminus B_O(R'_{j+1}))) 
    = \Omega\Big(e^{-\frac{R}{2}\big(-(1-c)(1-\alpha)^2 + (1-c)(2\alpha-1+\frac{1-\alpha}{2})+c\big)}\Big) \\
& \qquad = \Omega\Big(e^{-\frac{R}{2}\big((1-c)(\frac{7\alpha-3}{2}-\alpha^2)+c\big)}\Big).
\end{align*}
    Note that for $\frac12 < \alpha < 1$, we have $\frac{7\alpha-3}{2}-\alpha^2 < 1$, and thus the constant factor multiplying $-\frac{R}{2}$ in the exponent of the last term of the previously displayed equation is clearly bounded by $c < 1$. Hence,   $\mu(B_{Q}(R)\cap (B_O(R'_j) \setminus B_O(R'_{j+1})))=\Omega(e^{-\gamma R/2})$ for some $\gamma < 1$. Hence, the expected number of vertices inside $B_{Q}(R)\cap (B_O(R'_j) \setminus B_O(R'_{j+1}))$ is $n^{1-\gamma}$. By Chernoff bounds
the first part of the lemma follows. For the second part, by  Lemma~\ref{lem:muBallInterGen} (applied with
  $\rho_A=R$, $\rho_O=R/2$, and $R'_T < r_A=r_Q \leq R'_{T-1}$), since $r_Q > R/2$ still holds, we have
\[
\mu(B_{Q}(R)\cap B_O(R/2)) =\Theta(e^{-(\alpha - \frac12)\frac{R}{2}-\frac12 r_{Q}})
   = \Omega(e^{-(\alpha - \frac12)\frac{R}{2}-\frac12 R'_{T-1}}).
\]
     Note that $R'_{T-1}\leq \frac{R}{2}+R(1-c)(1-\alpha)$ must hold, as otherwise $R'_{T+1}=R'_{T-1}-R(1-c)(1-\alpha) > R/2$ would hold, and then $c - \frac{T+1}{2}(1-c)(1-\alpha)> \frac12$, contradicting the definition of $T$. 
Thus,
\[
   \mu(B_{Q}(R)\cap B_O(R/2)) =\Omega(e^{-\frac{R}{2} \big(\alpha + (1-c)(1-\alpha)\big)}).
\]
Since $ \alpha + (1-c)(1-\alpha) < 1$, $\mu(B_{Q}(R)\cap B_O(R/2)) =n^{1-\gamma'}$ for some $\gamma' > 0$, and by the same argument as before, the second part of the lemma follows.
\end{proof}
Intuitively, due to the nature of hyperbolic space, one expects that vertices in a center path from a vertex $v$ that is ``sufficiently far'' from the origin, tend to stay close to the ray between $v$ and the origin. One way of capturing this foreknowledge is to show that the angle subtended at the origin by the endpoints of any such center path is ``rather small''. The following lemma formalizes this intuition.
\begin{lemma}\label{lem:nointersectioncenterpath}
Let $u$ be a vertex with $r_u > R_{i_0}$ and let $u=:u_0,u_1,\ldots,u_m$ be a center path starting at $u$, or the longest subpath found in an attempt of finding a center path starting at $u$, in case no center path is found. Then,
\[
|\theta_{u} - \theta_{u_m}| =o(\log^{C_0} n/n).
\]
\end{lemma}
\begin{proof}
Let  $\ell$ be such that $R_{\ell}< r_{u} \leq R_{\ell+1}$, and note that $\ell \geq i_0$. 
Since by making paths longer, the regions exposed get only larger and differences in angles get larger as well, it suffices to consider the case where indeed a center path is found. Suppose there exists such a center path starting from $u$, that is, a sequence of vertices denoted by $u=:u_0,u_1,\ldots,u_m$, satisfying $d(u_j,u_{j+1}) \leq R$ for all $0 \leq j \leq m-1$, $R_{\ell-j} < r_{u_j} \leq R_{\ell-j+1}$ for any $0 \leq j \leq m-1$, and $R_{i_0-1} < r_{u_m} \leq R_{i_0}$. By Remark~\ref{rem:aproxAngle}, $|\theta_{u_j} - \theta_{u_{j+1}}| \leq 2e^{\frac12(R-R_{\ell-j}-R_{\ell-j+1})}(1+o(1))$, and thus, by the triangle inequality for angles, 
\begin{align*}
& |\theta_{u} - \theta_{u_m}| 
  \leq \sum_{j=0}^{m-1} |\theta_{u_j}-\theta_{u_{j+1}}|
   \leq 2(1+o(1))\sum_{j=0}^{m-1}e^{\frac12(R-R_{\ell-j}-R_{\ell-j+1})} \\
   & \quad = 2(1+o(1))\sum_{i=i_0+1}^{m+i_0}e^{\frac12(R-R_{i-1}-R_{i})}
           \leq 2(1+o(1))\sum_{i \geq i_0} 2e^{\frac12(R-R_{i}-R_{i+1})}.
\end{align*}
Since $R_i=Re^{-\alpha^i/2}=(1+o(1))R(1-\alpha^i/2)$ for $i \geq i_0$,
\begin{align*}
& |\theta_{u} - \theta_{u_m}| 
  \leq  (2+o(1))e^{-\frac12 R}\sum_{i\geq i_0}e^{(1+o(1))\frac{R}{4}\left( \alpha^{i}+\alpha^{i+1}\right)}
 = O(e^{-\frac12 R}R^{\frac12 (C_0+C_0\alpha)(1+o(1))})\\
& \quad = O(n^{-1}(\log n)^{\frac12 C_0(1+\alpha)(1+o(1))})
 = o(\log^{C_0} n/n),
\end{align*}
where the last equality is a consequence of $\alpha < 1$. The statement follows.
\end{proof}
We now have all the necessary ingredients to state and prove the main lemma.

\begin{algorithm*}[t]
\caption{Sequence in which regions are exposed}\label{alg:main}
\begin{algorithmic}[1]
%
\Procedure{Expose}{$Q,\xi$}
\State Let $\ell_0$ be the smallest integer such that $R_{\ell_0}>R-\xi$
\For{$j = 0,\ldots, C''\log n$}
\State Expose $A^{0}(R_{i_0})$
\State $\ell\gets \ell_0+1$
\Repeat\Comment{Start of $(j+1)$-th phase}
      \State $\ell \gets \ell-1$
      \If{$\ell = i_0$}
        \Return{'no path'} and stop \label{line:noPath}
      \EndIf
      \State Expose $A^{j+1}(R_{\ell})\setminus A^{j}(R_{i_0})$ \Comment{$(\ell_0-\ell+1)$-th sub-phase of the $(j+1)$-th phase}
\Until{$\big(\exists Q=:v_0,\ldots,v_{k_{j+1}}\in\calP$ a path, $v_i\in A^j(R_{i_0})\cup A^{j+1}(R_{\ell})$ for $i<k_{j+1}$, \newline \mbox{}\hspace{7em}$v_{k_{j+1}}\in A^{j+1}(R_{\ell})\setminus A^{j}(R_{i_0})$ and $r_{v_{k_{j+1}}} \leq R-\xi\big)$} \label{line:cond}
\If{$\exists$ center path starting from $v_{k_{j+1}}$}  
  \State \Return{'success'} and stop
\EndIf
\EndFor
\State \Return{'failure'} and stop
\EndProcedure
\end{algorithmic}
\end{algorithm*}

\begin{lemma}\label{lem:main}
Let $Q \in \calQ$ be a vertex added. 
With probability at least $1-o(n^{-3/2})$, $Q$ either connects to a vertex in $\calP\cap B_{O}(R_{i_0})$ through a path of length $O(\log^{C_0+1+o(1)} n)$, or any path starting at $Q$ has length at most $O(\log^{C_0+1+o(1)} n)$.
\end{lemma}
\begin{proof}
Suppose first that $r_Q \leq R_{i_0}$. By applying iteratively either  Lemma~\ref{lem:farcenter} and Lemma~\ref{lem:closecenter} (if $R_{j_0} < r_Q \leq R_{i_0}$), or only Lemma~\ref{lem:closecenter} (if $r_Q \leq R_{j_0}$), w.e.p., $Q$ is connected to a vertex in $B_O(R/2)$ through a path of length $O(\log R)+O(T)=O(\log \log n)$. 

Otherwise, suppose now $r_Q > R_{i_0}$. 
If $r_Q > R-{\xi}$ for some $\xi > 0$, then by Lemma~\ref{lem:nolongpath}, with probability at least $1-o(n^{-3/2})$, either all paths starting at $Q$ have length $O(\log n)$, or after at most $O(\log n)$ steps we reach a vertex $v$ with $r_v \leq R-\xi$. Since an additional $O(\log n)$ steps will be negligible in the following, we may thus assume $r_Q \leq R-\xi$. Let $\ell_0$ be the smallest
  integer $i$ such that $R_{i}> R-\xi$.

Let $\ell$ be such that $R_{\ell} < r_Q \leq R_{\ell+1}$ and note that $\ell \geq i_0$. 
By Lemma~\ref{lem:centerpath}, the probability of having no center path starting at $v_0=v_{k_0}:=Q$ is at most $c < 1$. If we find a center path starting at $v_{k_0}$, we have reached our goal and stop with success. Otherwise, we will expose regions around the location of $Q$ in different exposing phases. 
Specifically, fix $\varepsilon > 0$ to be any arbitrary small constant.
Then,  for $0\leq j\leq C''\log n$, where 
  $C''$ is a sufficiently large constant, define 
\begin{align*}
A^{j}(\rho):= 
  \{P \mid r_P>\rho, 
      |\theta_P-\theta_Q|\leq\tfrac{j+1}{n}\log^{C_0+\varepsilon}n \}
\end{align*} 
and perform the \textsc{Expose} procedure whose pseudocode is 
  given in Algorithm~\ref{alg:main} (see illustration in Figure~\ref{fig:expose}).
\begin{figure}
\begin{center}
\ifpdf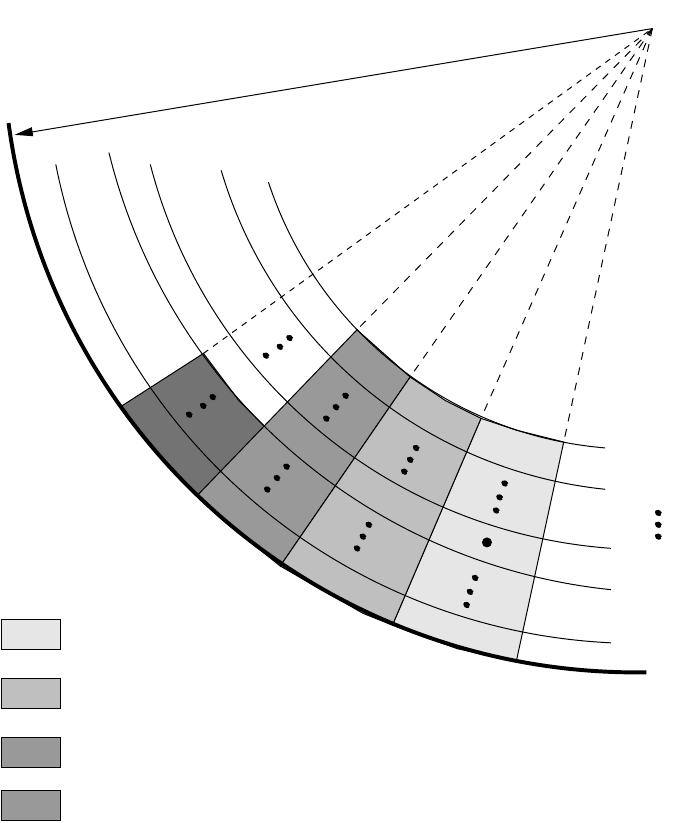\else\input{expose.pstex_t}\fi
\end{center}
\caption{Illustration of sequence of exposures of procedure \textsc{Expose}.}\label{fig:expose}
\end{figure}

If 'no path' is reported in Line~\ref{line:noPath} of \textsc{Expose}
  during phase $j+1$, then among the paths starting from $v_0$ with all 
  vertices $v_i$ satisfying $r_{v_i} > R_{i_0}$, there is no path with 
  its last vertex $v_{k_{j+1}} \in A^{j+1}(R_{i_0}) \setminus A^j(R_{i_0})$ satisfying $r_{v_{k_{j+1}}} \leq R-\xi$.
We will show now that the length of any such path is with
  probability at least $1-o(n^{-3/2})$ bounded by 
  $O(j\log^{C_0+\varepsilon}n)$: indeed, partition the disc of radius
  $R$ centered at the origin into $\Theta(n/\log n)$ equal sized
  sectors of angle $C' \log n/n$ for $C' > 0$ sufficiently large. By
  Chernoff bounds for Poisson variables (see for
  example~\cite[Theorem~A.1.15]{AlonSpencer}) together with a union
  bound over all $\Theta(n/\log n)$ sectors, with probability at least
  $1-o(n^{-3/2})$, the number of vertices in each sector is
  $\Theta(\log n)$.  
Thus, with this probability, the number of vertices in 
  $A^j(R_{i_0})$ is $O(j \log^{C_0+\varepsilon} n)$.
By~\eqref{eq:anglebound} 
  any two adjacent vertices $w,w'$ with $r_w,r_w' > R_{i_0}$
  satisfy $|\theta_{w}-\theta_{w'}| = O(\log^{C_0+o(1)} n/n)$.
If a path starting from $v_0$ with all vertices $v_i$ satisfying
  $r_{v_i} > R_{i_0}$ ends with some vertex 
  $v\in A^{j}(R_{i_0})$, then
  the path does not extend beyond $A^{j+1}(R_{i_0})$, and
  thus in this case, with probability at least $1-o(n^{-3/2})$, its 
  length is at most the number of vertices
  in $A^j(R_{i_0})$, i.e., $O(j \log^{C_0+\varepsilon} n)$. 
If the path ends
  with a vertex $v \in A^{j+1}(R_{i_0}) \setminus A^j(R_{i_0})$, then note that 
  $r_v > R-\xi$ must hold. 
Consider then the last vertex $w$ in this path with 
  $w \in A^j(R_{i_0})$. 
Since the next vertex $w'$ satisfies $w'  \in 
  A^{j+1}(R_{i_0}) \setminus A^j(R_{i_0})$ 
  and $r_w' > R-\xi$, by calculations as in~\eqref{eq:anglebound},
  it must hold that
  $|\theta_{w}-\theta_{w'}| = O(\log^{C_0/2+o(1)} n/n)$. Moreover, by
  Corollary~\ref{cor:nolongpath}, with probability at least
  $1-o(n^{-3/2})$, 
  any two vertices $w'',w'''$ on a path containing only
  vertices inside $A^{j+1}(R_{i_0}) \setminus A^j(R_{i_0})$ 
  satisfy $|\theta_{w''}-\theta_{w'''}| =
  O(\log n/n)$. Thus, also $|\theta_{w}-\theta_{v}| =
  O(\log^{C_0/2+o(1)} n/n)$. Since this would also have to hold for another
  possible neighbor $z$ of $v$ with $r_{z} > R_{i_0}$, and 
  $z \notin A^{j+1}(R_{i_0})$, 
  starting from $v$ the path cannot extend to a vertex
  outside $A^{j+1}(R_{i_0})$ as well. 
Hence, with probability at least
  $1-o(n^{-3/2})$, the number of vertices of such a path is bounded by
  $O((j+1) \log^{C_0+\varepsilon} n)$.

If on the other hand the condition of Line~\ref{line:cond} of 
  \textsc{Expose} is satisfied during phase $j$, then for some 
  $i_0\leq\ell'\leq\ell_0$ there is a
  path $v_0,\ldots,v_{k_{j+1}}\in\calP$ with 
  $v_i \in A^j(R_{i_0})\cup A^{j+1}(R_{\ell'})$
  for all $i< k_{j+1}$, 
  $v_{k_{j+1}}\in A^{j+1}(R_{\ell'}) \setminus A^j(R_{i_0})$.
Recall that $r_{v_{k_{j+1}}} \leq R-\xi$, 
  and observe that $R_{\ell'} < r_{v_{k_{j+1}}} \leq R_{\ell'+1}$. 
Consider then the restriction of a center path
  starting from $v_{k_{j+1}}$ to the region
\[
S_{j+1} := 
\{P \mid  \tfrac{j+1}{n}\log^{C_0+\varepsilon}n < |\theta_{P} - \theta_Q| \leq \tfrac{j+2}{n}\log^{C_0+\varepsilon} n\}.
\]
Note first that since $R_{\ell'} < r_{v_{k_{j+1}}} \leq R_{\ell'+1}$,  the first region to be exposed on the center path starting from $v_{k_{j+1}}$ contains vertices $v$ with $R_{\ell'-1} < r_v \leq R_{\ell'}$, and this region when restricted to $S_{j+1}$ has not been exposed before. In particular, the regions exposed by the center paths of $v_{k_i}$ and $v_{k_{j+1}}$, whenever restricted to $S_i$ and $S_{j+1}$, respectively, for $i < j+1$ are disjoint. Suppose without loss of generality 
that $|\theta_Q - \theta_{v_{k_{j+1}}}| \leq (j+\frac32)\log^{C_0+\varepsilon}n/n$. Observe that by  Lemma~\ref{lem:nointersectioncenterpath} the angle between any vertex $v$ with $r_v > R_{i_0}$ and the terminal vertex of the center path starting at $v$ or the longest subpath in an attempt of building a center path starting at $v$, in case a center path is not found, is bounded by $o(\log^{C_0} n/n)$. Therefore, at least half of the region to be exposed by a center path in each step is retained in $S_{j+1}$: more precisely (see Figure~\ref{fig:centerPath}), in the first step the region containing all points $P$ that were to be exposed in a center path without restrictions and that satisfy $|\theta_Q - \theta_P| \geq |\theta_Q - \theta_{v_{k_{j+1}}}|$ is included in $S_{j+1}$, and at every subsequent step on the center path starting from a vertex $u$, the region of all points $P$ that were to be exposed in a center path without restrictions and that satisfy $|\theta_Q - \theta_P| \geq |\theta_Q - \theta_{u}|$ is always included in $S_{j+1}$. Lemma~\ref{lem:centerpath} can still be applied, and the probability of having no center path starting at $v_{k_{j+1}}$ is at most $c' < 1$, and by disjointness, this is clearly independent of all previous events. 

\begin{figure}
\begin{center}
\ifpdf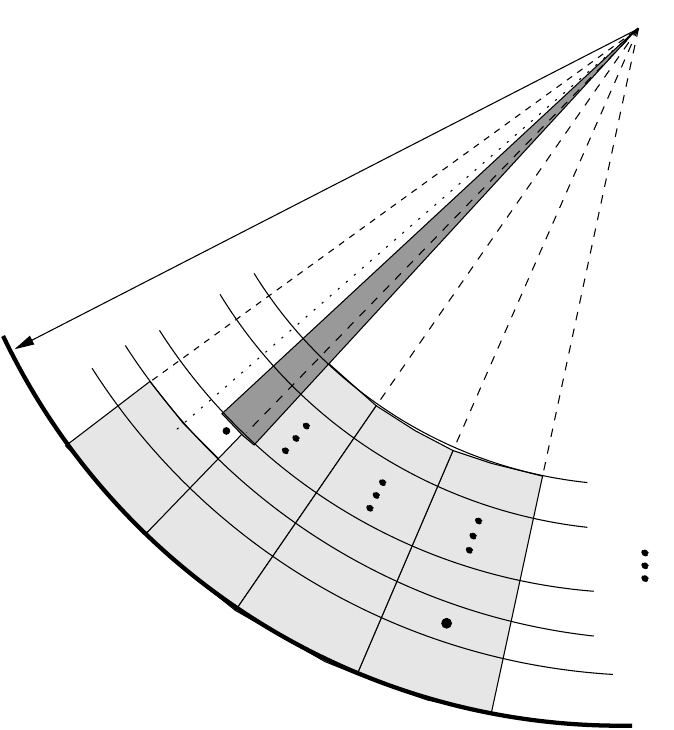\else\input{centerPath.pstex_t}\fi
\end{center}
\caption{The region that has already been exposed is lightly shaded. The $AOB$ slice of $B_{O}(R)$ corresponds to $S_{j+1}$. The darkly shaded area contains the region to be exposed in a center path without restrictions starting at $v_{k_3}$.}\label{fig:centerPath}
\end{figure}

The probability that the exposing process returns 'failure' is the probability that  during $C'' \log n$ phases  no center path is found. By independence, for $C''(c)$ sufficiently large, this probability is at most ${c'}^{C'' \log n}=o(n^{-3/2})$. Thus, with probability at least $1-o(n^{-3/2})$, the exposing process ends with 'no path' in some phase $j$, or with 'success' in some phase $j'$. If the exposing process ends with 'no path' in some phase $j$, with probability $1-o(n^{-3/2})$ any path starting at $v_0$ containing only vertices $v$ with $r_v > R_{i_0}$ has length at most $O((j+1) \log^{C_0+\varepsilon} n)=O(\log^{C_0+1+\varepsilon} n)$. If the exposing process ends with 'success' in some phase $j'$, a center path of length $O(\log R)$ starting from $v_{k_{j'}}$ is found, and as in the case of 'no path', the total length of the path starting at $v_0$ together with the center path is at most $O(\log^{C_0+1+\varepsilon} n)$. 
Thus, with probability at least $1-o(n^{-3/2})$, $v_0$ either connects 
  to a vertex in $\calP\cap B_O(R_{i_0})$ through a path of length $O(\log^{C_0+1+\varepsilon} n)$, or any path starting at $v_0$ has length at most $O(\log^{C_0+1+\varepsilon} n)$, and the statement follows, since $\varepsilon$ can be chosen arbitrarily small.
\end{proof}
The upper bound for the diameter now follows easily. 
\begin{theorem}\label{thm:main}
A.a.s., any two vertices $u$ and $v$ belonging to the same connected 
  component satisfy $d_G(u,v)=O(\log^{C_0+1+o(1)} n)$.
\end{theorem}
\begin{proof}	
Let $Q \in \calQ$ be a vertex added in the Poissonized model. 
By Lemma~\ref{lem:main}, with probability at least $1-o(n^{-3/2})$, 
  either $Q$ connects to a vertex in $B_O(R/2)$ 
  using $O(\log^{C_0+1+o(1)} n)$ steps, or 
  all paths starting from $Q$ have 
  length  $O(\log^{C_0+1+o(1)} n)$. 
Note that all vertices inside $B_O(R/2)$ form a clique, and by 
  Lemma~\ref{lem:muBall} together with Chernoff bounds there 
  are w.e.p.~$\Theta(n^{1-\alpha})$ vertices inside $B_O(R/2)$. 
By de-Poissonizing the model, using~\eqref{lem:Poisson}, the same 
  statement holds with probability at least $1-o(n^{-1})$ for our particular 
  choice of $Q$ in the uniform model. 
Using a union bound over all $n$ vertices, with probability $1-o(1)$ the 
  same holds for all vertices simultaneously.  
Thus, the maximum graph distance between any two vertices of the same 
  connected component is $O(\log^{C_0+1+o(1)}n)$,
  and the statement follows.
\end{proof}

By the result of Bode, Fountoulakis and M\"{u}ller~\cite{BFM13}, a.a.s., all 
  vertices belonging to $B_O(R/2)$ are part of a giant component. 
We call this the \emph{center giant component}, and we will show that 
  there is no other giant component. 
\begin{corollary}\label{cor:main3}
A.a.s., the size of the second largest component of $G_{\alpha,C}(n)$ is 
  $O(\log^{2C_0+1+o(1)}n)$.
\end{corollary}
\begin{proof}
Having de-Poissonized some results in the previous proof, we may now work directly in the model $G_{\alpha,C}(n)$ during this proof.
As in the proof of Lemma~\ref{lem:main}, partition the disc of radius $R$ centered at the origin into $\Theta(n/\log n)$ equal sized sectors of angle $C' \log n/n$ for some large constant $C' > 0$ and order them counterclockwise. The distance between sectors is measured according to this ordering: two sectors are at distance $D$ from each other, if the shorter of the two routes around the clock between the two sectors has exactly $D-1$ sectors in between.

Now, by the proof of
  Theorem~\ref{thm:main}, a.a.s., every vertex $v$ with 
  $r_v \leq R_{i_0}$ connects to $B_O(R/2)$, so a.a.s., every such vertex 
  is part of the center giant component. 
We may assume that  all other vertices $v$ in other components satisfy 
  $r_v > R_{i_0}$, and consider only these components from now on. 
By Theorem~\ref{thm:main}, a.a.s.,~for any two vertices of 
  the same component, their graph distance is bounded by 
  $O(\log^{C_0+1+o(1)} n)$.  
Recall that for any two vertices $v_i,v_j$ with $d(v_i,v_j) \leq R$ 
  and $r_{v_i},r_{v_j} > R_{i_0}$, by~\eqref{eq:anglebound}, 
  $|\theta_{v_i} - \theta_{v_j}| = O((\log n)^{C_0+o(1)}/n),$ and 
  therefore, by the triangle inequality for angles, for 
  any two vertices $u$ and $v$ belonging to the same component, a.a.s.\ we have
  $|\theta_{u} - \theta_{v}| = O((\log n)^{2C_0+1+o(1)}/n).$ 
Since for any two vertices $u$ and $v$ belonging to sectors at distance 
  $2 \leq D=o(n/\log n)$ from each other we have 
  $|\theta_{u}-\theta_{v}|=\Theta(D \log n/n)$, it must hold that 
  between any two vertices $u$ and $v$ of the same component, there 
  are at most $O(\log^{2C_0+o(1)} n)$ sectors. 
By Chernoff bounds for binomial random variables together with a
  union bound over all sectors, for $C' > 0$ large enough, a.a.s., in
  any sector there are at most $O(\log n)$ vertices.  Thus,
  a.a.s.,~the number of vertices of the second largest component is
  bounded by $O(\log^{2C_0+1+o(1)} n)$.
\end{proof}

\section{The lower bound on the diameter}\label{sec:lower}
In this section we show that the diameter of a random
  hyperbolic graph is $\Omega(\log n)$. In fact, we do more, since we
  actually establish the
  existence of a component forming a path of length $\Theta(\log n)$.

To achieve this section's main objective we 
  rely on additional bounds concerning the measure
  of regions defined in terms of set algebraic manipulation of distinct balls.
In the following section we derive these bounds. The reader, however, might prefer to skip their proofs 
  upon first reading, and come back to them once their intended use
  is understood.

\subsection{Useful bounds}%
The vertices of 
  the path component whose existence we will establish will be shown to 
  be at distance $R-\Omega(1)$ of the origin $O$. 
This explains why below we focus on approximating the measure of regions
  which are close to the boundary.

Our next result establishes two facts. 
First, it gives an approximation for the measure of a band centered
  at the origin. 
Then, it shows that (in expectation) 
  there are $\Omega(1)$ vertices in the intersection of two constant
  width bands of radius $R-\Theta(1)$, 
  one of which is centered at the origin and the other one 
  centered at a point that is inside the previous band. 

\begin{figure}[h]
\begin{center}
\ifpdf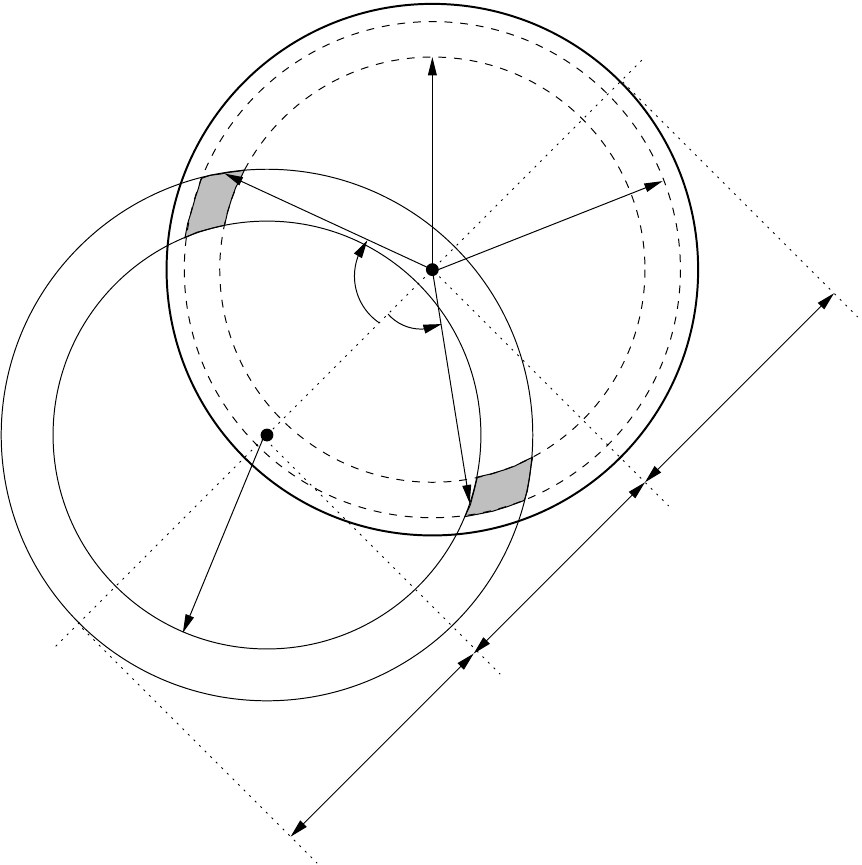\else\input{intersectBands.pstex_t}\fi
\end{center}
\caption{the shaded area corresponds to 
  $\big(B_A(R)\setminus B_A(R{-}c_3)\big)\cap 
    \big(B_{O}(R{-}c_1)\setminus B_{O}(R{-}c_2)\big)$
  (also, $\theta=\theta_{R}(r,r_A)$ and $\theta'=\theta_{R-c_3}(r',r_A)$)}\label{fig:interBands}
\end{figure}

\begin{lemma}\label{lem:medidalb} 
The following statements hold: 
\begin{enumerate}[leftmargin=*]
\item\label{it:medidalb1}
If $0 \leq \rho'_O \leq  \rho_O < R$, then
\[
\mu(B_{O}(\rho_{O})\setminus B_{O}(\rho'_{O})) 
  = e^{-\alpha(R-\rho_{O})}(1-e^{-\alpha(\rho_{O}-\rho'_{O})}+o(1)).
\]

\item\label{it:medidalb2}
If $0 < c_1  < c_2$ are two positive constants, 
  $R-c_2 \leq r_A \leq R-c_1$, and $c_3=\Omega(1)$, then
\[
\mu((B_A(R)\setminus B_A(R{-}c_3))\cap (B_{O}(R{-}c_1)\setminus B_{O}(R{-}c_2)))  = \Omega(e^{-R/2}).
\]
\end{enumerate}
\end{lemma}
\begin{proof}
The first part of the lemma follows directly from Lemma~\ref{lem:muBall}.
For the second part (see Figure~\ref{fig:interBands}), again relying on Lemma~\ref{lem:aproxAngle}  we get that the second
  expression equals
\begin{align*}
& 2\int_{R-c_2}^{R-c_1}\int_{\theta_{R-c_3}(r,r_A)}^{\theta_{R}(r,r_A)} \frac{f(r)}{2\pi}d\theta dr 
  = 2\frac{e^{(R - r_A)/2}(1-e^{-c_3/2})}{\pi C(\alpha,R)} 
 \times\int_{R-c_2}^{R-c_1} e^{-r/2}\alpha\sinh(\alpha r)\Big(1+O(e^{-r})\Big)dr \\
& \ = \frac{\Omega(1)}{C(\alpha,R)}\Big(e^{(\alpha-\frac12)R}
    + O(e^{(\alpha-3/2)R})\Big).
\end{align*}
The desired conclusion follows recalling that $C(\alpha,R)=\cosh(\alpha R)-1=
    \frac{1}{2}e^{\alpha R}+O(e^{-\alpha R})$.
\end{proof}
The following result implicitly establishes that 
  provided $c_1$, $c_2$, and $c_3$ are appropriately chosen constants,
  if $A$ is a vertex such that $R-c_2<r_A<R-c_1$, then
  in expectation $A$ has $\Omega(1)$ neighbors at distance
  at least $R-c_3$ in the 
  band $B_{O}(R-c_1)\setminus B_{O}(R-c_2)$ both in the clockwise
  and anticlockwise direction.
\begin{lemma}\label{lem:BallIntersection}
Let $0 < c_3 < c_1 < c_2$ be small positive constants so 
  that $2e^{c_1-c_2} > e^{c_3/2}$ holds.
Suppose $R-c_2 \leq r_{A},r_{B} \leq R-c_1$ and  $R-c_3 \leq d(A,B) \leq R$.  
Then 
\[
\mu\big([(B_B(R)\setminus B_B(R{-}c_3)) \\
 \cap 
  (B_{O}(R{-}c_{1})\setminus B_{O}(R{-}c_{2}))]\setminus B_A(R)\big) \\
 =\Omega(e^{-R/2})=\Omega(1/n).
\]
\end{lemma}
\begin{proof}
By Lemma~\ref{lem:medidalb}, Part~\ref{it:medidalb2},
  we know that $B_{B}(R)\setminus B_{B}(R-c_{3})$
  intersects the band at distance between $R-c_{2}$ and $R-c_{1}$, 
  i.e.~$B_{O}(R-c_{1})\setminus B_{O}(R-c_{2})$, in a region of 
  measure $\Omega(e^{-R/2})$. 
Note that the intersection comprises two disconnected regions of equal 
  measure, say $\cal{D}$ and $\cal{D}'$. We may assume that $A \in \cal{D}$ and also that for all points $P\in \cal{D}$ and 
  $P' \in \cal{D}'$ we have 
  $\theta_P < \theta_B < \theta_{P'}$. 
  We will show that $B_A(R)$ does not intersect $\cal{D}'$: 
  suppose for contradiction that
  $P \in \cal{D}$ and
  $P' \in \cal{D}'$ are within distance $R$: 
  then, by  Remark~\ref{rem:aproxAngle}, 
  we would have 
\begin{equation}\label{eq:supposedbound}
|\theta_P - \theta_{P'}|
  \leq (2+o(1))e^{\frac12(R-2(R-c_2))}
  = (2+o(1))e^{-R/2}e^{c_2}.
\end{equation} 
On the other hand, for any 
  $P\in\cal{D}$, since $d(P,B)\geq R-c_3 > R-c_1 > R-c_2$, 
  we have again by  Remark~\ref{rem:aproxAngle},
  $|\theta_B - \theta_P|\geq (2+o(1))e^{\frac12(R-c_3-2(R-c_1))}=(2+o(1))e^{-R/2}e^{c_1-\frac{c_3}{2}}$, and the same bound holds for
    $|\theta_B - \theta_{P'}|$. 
Since $P$ and $P'$ satisfy $\theta_{P} < \theta_B < \theta_{P'}$, 
  we have $|\theta_{P}-\theta_{P'}|=|\theta_B-\theta_P|+|\theta_B-\theta_{P'}|$, 
  and thus $|\theta_P-\theta_{P'}|\geq (4+o(1))e^{-R/2}e^{c_1-\frac{c_3}{2}}$. 
Since by assumption $2e^{c_1-\frac{c_3}{2}} > e^{c_2}$, this 
  contradicts~\eqref{eq:supposedbound}. 
The lemma follows.
\end{proof}

\subsection{The existence of a path component of length $\Theta(\log n)$}%
The main objective of this section is to show 
  that inside a band at constant radial distance from the
  boundary we find $\Theta(\log n)$ vertices forming a path. 
Although the calculations involved require careful bookkeeping,
  at a high level the proof strategy is not complicated.
We now informally describe it.
We fix a band centered at the origin, say 
  $B_{O}(R-c_1)\setminus B_{O}(R-c_2)$, where $c_1$ and $c_2$ are constants.
Roughly, we show there are 
  $m=o(\sqrt{n})$ vertices $Q_1,\ldots, Q_m$ in the aforementioned band
  that satisfy the following two properties: 
  (1) the $Q_i$'s are spread out throughout the band,
  i.e., each pair subtends an angle at the origin that is ``sufficiently''
  large, and (2) with ``not too small" probability 
  each $Q_i$ is the endvertex of a path 
  of length $\Theta(\log n)$ which is completely 
  contained in the mentioned band. 
The former property implies that the events considered in the latter 
  are independent, from where it easily follows that 
  w.e.p.~there must be a path of length $\Theta(\log n)$ inside the band.
However, it will become clear next that there are subtle issues that 
  must me dealt with carefully in order to formalize this paragraph's
  discussion.

\begin{theorem}\label{thm:main2}
A.a.s., there exists a component forming a path of length $\Theta(\log n)$.
\end{theorem}
\begin{proof}
As before, we work first in the Poissonized model and derive in the end from it the result in the uniform model.
Fix throughout the proof $c_1, c_2, c_3$ three positive constants 
  such that $c_3 < c_1 < c_2$ and $2e^{c_1-c_2} < e^{-c_3/2}$.
First, by Lemma~\ref{lem:medidalb}, Part~\ref{it:medidalb1},
  applied with $\rho_0=R-c_1$ and $\rho'_0=R-c_2$ we have
\[
\mu(B_{O}(\rho_{O})\setminus B_{O}(\rho'_{O}))
    = e^{-\alpha c_1}(1-e^{-\alpha(c_2-c_1)}+o(1)) 
    = \Theta(1).
\]
Let $\Theta\subseteq [0,2\pi)$ be a set of forbidden angles 
  such that $\mu(R_{\Theta})<1$, where 
  $R_{\Theta}:=\{(r_P,\theta_P): 0 \leq r_P < R, \theta_P \in \Theta\}$
  (for a geometric interpretation of $R_{\Theta}$, note that 
  when $\Theta$ is an interval, $R_{\Theta}$ is a cone with vertex $O$).
As a constant fraction of
  the angles is still allowed, clearly,
\begin{equation}~\label{eq:crit1}
\mu(\big((B_{O}(\rho_{O})\setminus B_{O}(\rho'_{O}))\setminus R_{\Theta}\big)=\Theta(1)
\end{equation}
still holds.
For any vertex $A$ with 
  $R - c_2 \leq r_A \leq R-c_1$, by Lemma~\ref{lem:muBallInterGen}
  (applied with $\rho_A=R$ and $\rho_{O}=R$)
\begin{equation}\label{eq:crit2}
\mu(B_{A}(R)\cap B_{O}(R)) 
  = \frac{2\alpha}{\pi(\alpha-\frac12)}
        e^{-\frac{1}{2}r_A}
        + O(e^{-\alpha r_{A}}) 
 =\Theta(1/n), 
\end{equation}
which together with   
  Lemma~\ref{lem:medidalb}, Part~\ref{it:medidalb2}, yields
\begin{equation}\label{eq:crit3}
\mu(\big(B_A(R)\setminus B_A(R{-}c_3)\big) \cap 
  \big(B_{O}(\rho_{O})\setminus B_{O}(\rho'_{O})\big))
  = \Theta(e^{-R/2})=\Theta(1/n). 
\end{equation} 
For two vertices $A,B$ with $R-c_3 \leq d(A,B) \leq R$ that satisfy 
  $\rho'_{O}\leq r_A, r_B \leq \rho_{O}$, by Lemma~\ref{lem:BallIntersection}, 
  together with~\eqref{eq:crit2},
\begin{equation}\label{eq:crit4}
\mu\big([(B_B(R)\setminus B_B(R{-}c_3)) \cap  
  (B_{O}(\rho_{O})\setminus B_{O}(\rho'_{O}))]\setminus B_A(R)\big) \\
 =\Theta(e^{-R/2})=\Theta(1/n).
\end{equation}
Let $\varepsilon=\varepsilon(\alpha)$ be a constant chosen small enough so that $1-\frac{1}{2\alpha}-\varepsilon > 0$. 
Let $\nu > 0$ be a sufficiently small constant. Let $m:=\nu n^{1-\frac{1}{2\alpha}-\varepsilon}$, and note that since $\alpha < 1$, we have $m=o(n^{1/2})$. We will add in the following, if necessary, up to $m$ vertices $Q_1,\ldots,Q_m \in \calQ$ to $\calP$, all of them chosen independently and following the same distribution as in $G_{\alpha,C}(n)$. 
In order to be more precise, define for $1 \leq j \leq m$, 
  the event $\calE_{j}$ that occurs when
  following conditions hold (if one condition fails, then stop exposing and checking further conditions and proceed with the next $j$; also stop if all conditions hold for one $j$):
%
%
\begin{itemize}
\item \textbf{Condition 1:} 
Add the $j$-th vertex $Q_j \in \calQ$ and let $A_0^j:=Q_j$.  We require $\rho'_O \leq r_{Q_j} \leq \rho_O$, and for $j \geq 2$, we additionally require the coordinates of $A_0^j$ and $A_0^k$ for any $1\leq k<j$ to be sufficiently 
  different in terms of their angles, i.e.,~letting 
  $\Theta_k=\{\theta: |\theta_{A_0^k}-\theta|\leq C''n^{-(1-\frac{1}{2\alpha}-\varepsilon)}\}$ for some large constant $C'$,
 we require that $\theta_{A_0^j} \not\in \cup_{k=1}^{j-1}\Theta_{k}$.

\item \textbf{Condition 2:} 
Expose the region $\big(B_{A_{0}^{j}}(R)\setminus B_{A_{0}^{j}}(R-c_{3})\big) 
  \cap \big(B_{O}(\rho_{O})\setminus B_{O}(\rho'_{O})\big)$. We require exactly one vertex in this region, call it $A_1^j$. Let $L:=\nu'\log n$ with $\nu'>0$  being a small constant. Then, inductively, for $1 \leq i < L$, 
  expose the region $\big(B_{A_{i}^{j}}(R)\setminus B_{A_{i}^{j}}(R-c_{3})\big) \cap \big(B_{O}(\rho_{O})\setminus B_{O}(\rho'_{O})\big) \setminus B_{A_{i-1}^j}(R)$. We
  require also exactly one vertex in this region, and name it inductively $A_{i+1}^j$. In other words, $A_{i+1}^j$ belongs to the $(R-c_{3},R)$-band 
  centered at $A_{i}^{j}$ and also to the $(\rho'_{O},\rho_{O})$-band
  centered at the origin $O$. For $i=L$, expose the region $\big(B_{A_{L}^{j}}(R)\setminus B_{A_{L}^{j}}(R-c_{3})\big) \cap \big(B_{O}(\rho_{O})\setminus B_{O}(\rho'_{O})\big) \setminus B_{A_{L-1}^j}(R)$, and we require that there is no more vertex in this region.
\item \textbf{Condition 3:} 
For any $0 \leq i \leq L$, expose $B_{A_{i}^{j}}(R) \cap (B_O(R) \setminus (B_O(R(1-\frac{1}{2\alpha}-\varepsilon)))$, and we require that $A_i^j$ has no other neighbor inside this region except the one(s) from the previous condition. 
\end{itemize}
We will now bound from above the probability that for all $1\leq j\leq m$ the events ${\calE}_{j}$ fail. 
  Note first that for $\nu$ sufficiently small, the set of forbidden angles
  $\Theta^{j}:=\cup_{k=1}^{j-1}\Theta_{k}$ still is such that $\mu(R_{\Theta^{j}})<1$, and~\eqref{eq:crit1} can 
  be applied with $R_{\Theta}=R_{\Theta^{j}}$ for 
  any $1\leq j\leq m$.   Hence, for any~$j$, independently of the outcomes of previous events, there is an absolute constant $c> 0$ such that for this~$j$ 
  the probability that Condition~1 holds is at least~$c$. 
Given that Condition~1 holds,
  by~\eqref{eq:crit3} applied to $A_0^j$ and~\eqref{eq:crit4} 
  applied successively to $A_1^j,\ldots,A_L^j$ we get that
  the probability that Condition~2 holds 
  is at least $c'^{-L}$ for some fixed $0 < c' < 1$. 
Suppose then that Condition~2 also holds.  
By a union bound and by~\eqref{eq:crit2}, we have
\[
\mu(\bigcup_{i=0}^{L} [B_{A_i^j}(R)\cap (B_{O}(R) \setminus B_O(R(1-\tfrac{1}{2\alpha}-\varepsilon))]) 
 \leq \sum_{i=0}^{L}\mu(B_{A_i^j}(R)\cap B_{O}(R))=O(L/n),
\]
and thus for $\nu'$ sufficiently small the probability that 
  Condition~3 holds is at least $n^{-\eta}$ for some fixed value 
  $\eta$ which can be made small by choosing $\nu'$ small 
(in fact, part of the region might
have been already exposed in Condition~2, but since we know there are
no other vertices in there, this only helps).  Altogether, we obtain
  \[
\Pr{({\calE}_{j}) \geq n^{-\eta'}}
\]
  for some $\eta' > 0$. Again, $\eta'$ can be made sufficiently small by making the constant $\nu'$ (and thus $L=\nu' \log n$) sufficiently small.

Now, since $d(A_{i-1}^j,A_i^j) \leq R$ and  $r_{A_{i-1}^j},r_{A_i^j} \geq R-c_2$, by Remark~\ref{rem:aproxAngle}, we have 
$|\theta_{A_{i-1}^j}-\theta_{A_i^j}| = O(e^{-R/2})$, and therefore $|\theta_{A_0^j}-\theta_{A_{L}^j}|=O(\log n/n)$. Since by construction for $j \neq j'$ we have
$|\theta_{A_0^j}-\theta_{A_0^{j'}}| = \omega(\log n/n)$, for $j \neq j'$ the 
  regions exposed in Condition~2 are disjoint. Also, for any $j \neq j'$, 
  if the region in Condition~3 containing one of the vertices $A_i^j$ or $A_{i'}^{j'}$ is exposed, then 
$B_{A_{i}^{j}}(R) \cap (B_O(R) \setminus (B_O(R(1-\frac{1}{2\alpha}-\varepsilon)))$ is disjoint from $B_{A_{i'}^{j'}}(R) \cap (B_O(R) \setminus (B_O(R(1-\frac{1}{2\alpha}-\varepsilon)))$: indeed, note that
 for any point $P \in B_{A_{i}^{j}}(R) \cap (B_O(R) \setminus (B_O(R(1-\frac{1}{2\alpha}-\varepsilon)))$ we have by Remark~\ref{rem:aproxAngle}
\[
|\theta_{A_i^j} - \theta_{P}| \leq (2+o(1))e^{\frac12 (R-(R-c_2)-R(1-\frac{1}{2\alpha}-\varepsilon))}
  = (2+o(1))e^{\frac12 c_2}n^{-(1-\frac{1}{2\alpha}-\varepsilon)}.
\]
By construction, we have $|\theta_{A_0^j}-\theta_{A_0^{j'}}|\geq C'n^{-(1-\frac{1}{2\alpha}-\varepsilon)}$ for some large enough $C' > 0$, and thus by the triangle inequality for angles there exists $C'':=C''(C') > 0$ such that 
  $|\theta_{A_i^j}-\theta_{A_{i'}^{j'}}|\geq (1+o(1))C''n^{-(1-\frac{1}{2\alpha}-\varepsilon)}$ holds for any $i,i'$ and any $j \neq j'$. Hence, the regions exposed in Condition~3 are disjoint, and by the same reason, the region exposed in Condition~2 for some $j$ and the one exposed in Condition~3 for some $j' \neq j$ are disjoint as well.
It follows that the probabilities of the corresponding conditions 
  to hold are thus independent. 
Hence, for $\nu'$ sufficiently small, $\eta'$ is small enough such that $-\eta'+1-\frac{1}{2\alpha} - \varepsilon > 0$, and
\[
\prod_{j=1}^{m} \Pr\big({\setCompl{\calE_j}}\big)\leq (1-n^{- \eta'})^{m} 
  = e^{- \Omega(n^{\xi})},
\]
for some positive $\xi > 0$.
Thus, w.e.p.~there exists one $j$, for which the event 
  $\calE_{j}$ holds. 

 In order to de-Poissonize, let $\calE_P^j$ be the event that in the Poissonized model there exists some $j$ for which the event $\calE_j$ holds, and similarly $\calE_U^j$ the corresponding event for the uniform model. Since we have shown that 
 $\Pr{(\calE_P^j)} =1-e^{-\Omega(n^{\xi})}$,
 and since $m=o(n^{-1/2})$, using~\eqref{lem:Poisson},
  we also have for some function $\omega_n$ tending to infinity arbitrarily slowly
$ \Pr{(\calE_U^j)} \geq 1-\omega_n\sqrt{n}e^{- \Omega(n^{\xi})}=1-e^{-\Omega(n^{\xi})}.
 $
 Let $\rho_O=R(1-\frac{1}{2\alpha}-\varepsilon)$. Denote by $N_U$ the random variable counting the number of vertices inside $B_O(\rho_O)$ in $G_{\alpha,C}(n)$. By Lemma~\ref{lem:muBall}, $\mu(B_O(\rho_O))=O(e^{-R/2-\alpha \varepsilon R})=o(1/n)$, and thus $\mathbb{E} N_U=o(1)$, and by Markov's inequality,
 $\Pr{(N_U =0)}=1-o(1).$
 Since
\[
1-o(1)= \Pr{(N_U = 0)} 
  = \Pr{(N_U=0 | \calE_U^j)}\Pr{(\calE_U^j)}+\Pr{(N_U=0 | \setCompl{\calE_U^j})}\Pr{(\setCompl{\calE_U^j})} 
  = (1-o(1))\Pr{(N_U=0 | \calE_U^j)}+o(1),
\]
we have $\Pr{(N_U=0 | \calE_U^j)}=1-o(1)$. Thus, 
 $\Pr{(N_U=0, \calE_U^j)}=1-o(1)$. If this event holds, this means that there is no vertex inside $B_O(\rho_O)$, and the vertices $A_0^j,\ldots, A_L^j$ form an induced path, yielding the desired result.
\end{proof}

\section{Conclusion}
We have shown that in random hyperbolic graphs a.a.s.~the diameter 
  of the giant component is $O(polylog(n))$, 
  the size of the second largest component~$O(polylog(n))$, 
  and at the same time there exists a path component of 
  length $\Theta(\log n)$. 
It is an interesting and challenging problem to 
  tighten these bounds by improving the exponents of the 
  $polylog(n)$ terms established in this work.

\section*{Acknowledgement}
We thank Tobias M\"uller for generously answering questions
  concerning the random hyperbolic graph model, specifically about the 
  state-of-art as well as pointers to the relevant 
  literature concerning the model.

\bibliographystyle{alpha}
\bibliography{biblio}


\end{document}